\documentclass[10pt,reqno]{amsart}

\usepackage[english]{babel}

\usepackage[margin=24mm]{geometry}
\usepackage[pagebackref,colorlinks,urlcolor=blue,citecolor=blue,linkcolor=blue]{hyperref}

\usepackage{amsmath}
\usepackage{amsfonts}
\usepackage{amsthm}
\usepackage{amssymb}

\usepackage[all]{xy}
\usepackage{verbatim}
\usepackage{a4wide}

\usepackage{hyperref}

\usepackage{enumerate} 
\usepackage{mathtools}
\numberwithin{equation}{section}
\usepackage{caption}
\usepackage{graphicx}
\graphicspath{ {./image/} }
\usepackage{tikz-cd}
\usepackage{xcolor}
\usepackage{difftrees}

\usepackage{shuffle}
\usepackage{comment}

\theoremstyle{definition}
\newtheorem{definition}{Definition}[section]
\newtheorem{example/}[definition]{Example}

\theoremstyle{plain}
\newtheorem{theorem}[definition]{Theorem}
\newtheorem{proposition}[definition]{Proposition}
\newtheorem{cor}[definition]{Corollary}
\newtheorem{lemma}[definition]{Lemma}

\theoremstyle{remark}
\newtheorem{remark}[definition]{Remark}

\usepackage[parfill]{parskip}    
\usepackage{graphicx}
\usepackage{amssymb}
\usepackage{epstopdf}
\usepackage{pdfpages}
\usepackage{hyperref}
\usepackage{url}
\usepackage{tikz-cd}
\newtheorem{note}[definition]{Note}

 \usepackage{stmaryrd}

    \newcommand*{\LTS}{\operatorname{LTS}}
        \newcommand*{\LAT}{\operatorname{LAT}}

\newcommand*{\Alg}{\operatorname{Alg}}
\newcommand*{\PLY}{\operatorname{PLY}}
\newcommand{\ord}[1]{\overrightarrow{#1}}
\newcommand*{\OT}{\ord{\operatorname{Tree}}}

\newcommand*{\T}{{\mathcal{T}}}

\newcommand{\I}{\mathcal{I}}
\newcommand{\noarg}{\,\_\,}
\newcommand{\OF}{\operatorname{OF}}

\title{The Connection Algebra of Reductive Homogeneous Spaces}
\author{Jonatan Stava}

\address{University of Bergen, Department of Mathematics, P.O.~Box 7803, 5020 Bergen, Norway}
\email{jonatan.stava@uib.no}

\subjclass[2020]{53C05, 17D99, 53C30, 05C05}

\keywords{Connection, reductive homogeneous space, connection algebra, Lie-Yamaguti algebra, rooted tree}

\begin{document}
\maketitle

\begin{abstract}

Consider the smooth sections of the tangent bundle of a reductive homogeneous space. This is a vector space over the field of real numbers. The canonical connection acts as a linear binary operator on this vector space, making it an algebra. If we include another binary operator defined as the negative of the torsion, the resulting algebraic structure is a post-Lie-Yamaguti algebra. This structure is closely related to Lie-Yamaguti algebras.
\end{abstract}


\section{Introduction}

The vector fields on a smooth manifold together with an affine connection form an algebraic structure we call the connection algebra. Understanding this structure has proven to be an important tool for doing numerical integration on manifolds. In the case of a locally flat and torsion free space, this algebra forms a pre-Lie algebra, an algebra first described by Gerstenhaber \cite{Gerstenhaber_63} and by Vinberg \cite{Vinberg}. It has later been shown to be closely related to B-series \cite{B-series_CHV_2010,CALAQUE_2011}, which is a vital tool for numerical integration \cite{HLW_2006_Book}. 

On any Lie group there exists a connection called the (-)-connection, which was one of the connections studied by Cartan in \cite{Cartan1927}. This connection has trivial curvature and parallel torsion, and in fact any manifold with a connection with these qualities can be locally represented as a Lie group with its (-)-connection \cite{Nomizu_54}. The connection algebra of a Lie group with its (-)-connection is a post-Lie algebra, a rich algebraic structure defined by Vallette in \cite{BV_2007} by means of operads. Post-Lie algebras also appear in the D-algebra described in \cite{MK_Wright_08} and have since become a key tool for numerical integration on Lie groups \cite{MK_Lundervold_2013,KMKL_15,CEFMK19}. Later research has also linked post-Lie algebras to geometries beyond that of a Lie group \cite{MK_Stern_Verdier_20,Gavrilov_paper_2022,grong2023postlie}. A smooth manifold can be represented locally as a symmetric space if and only if there exist a torsion free connection with parallel curvature on that manifold \cite{Nomizu_54}. The connection algebra on a symmetric space equipped with this connection is a Lie admissible triple algebra, a structure newly described in \cite{LAT_MKS_2023} that coincides with the notion of a G-algebra which was defined in \cite{GS_00,Sokolov_17} in a different context.

Building on the success of post-Lie and pre-Lie algebras, this article aims to provide a complete overview of the connection algebras related to reductive homogeneous spaces. Locally flat and torsion free spaces, Lie groups and symmetric spaces are all special cases of reductive homogeneous spaces. In his 1954 article \cite{Nomizu_54}, Nomizu provided a canonical connection for any reductive homogeneous space. Using this connection with the geometric constraints that follows, we describe the resulting algebra which we call a \textit{post-Lie-Yamaguti algebra} (PLY). This is a type of algebra that naturally generalizes both post-Lie algebras and Lie admissible triple algebras. The main result of this article is a description of the free PLY algebra, which can be found in Theorem \ref{th:main}.

In 1958 Yamaguti described an algebraic structure associated to reductive homogeneous spaces \cite{Yamaguti1958}. These algebras were later named Lie-Yamaguti algebras \cite{Kinyon_Weinstein_2001}. Post-Lie-Yamaguti algebras relate to Lie-Yamaguti algebras similar to how post-Lie algebras relate to Lie algebras and how Lie admissible triple algebras relate to Lie triple systems.

This paper is organized as follows: In Section \ref{sec:1} we introduce the concept of connection algebra. In Section \ref{sec:1.2} we review the types of connection algebras that have already been described in the literature. In Section \ref{sec:2} we present the connection algebra associated to reductive homogeneous spaces. Section \ref{sec:3} and \ref{sec:4} discusses the free algebras over one and two binary operators respectively, before a basis for the free post-Lie-Yamaguti algebra is presented in Section \ref{sec:5}. In Section \ref{sec:6} we look at Lie-Yamaguti algebras that are closely related to post-Lie-Yamaguti algebras.

\section{Connections with Parallel Curvature and Torsion}\label{sec:1}

Let $M$ be a real smooth manifold and denote by $\frak{X}_M = \Gamma(TM)$ the vector space of smooth sections of the tangent bundle, or the smooth vector fields. An affine connection $\nabla$ on $M$ is an $\mathbb{R}$-bilinear map $\nabla: \frak{X}_M \times \frak{X}_M \rightarrow \frak{X}_M$ satisfying
\[
\begin{split}
    \nabla_{fX}Y &= f\nabla_X Y, \\
    \nabla_X fY &= (Xf)Y + f \nabla_X Y,
\end{split}
\]
for all vector fields $X,Y\in \frak{X}_M$ and all smooth real valued functions $f\in C^{\infty}(M,\mathbb{R})$. It follows from this definition that the vector space $\frak{X}_M$ together with the $\mathbb{R}$-bilinear product $\nabla$ is an algebra which we call the \textit{connection algebra} of $(M,\nabla)$. To every connection on $M$ one can associate its curvature $R$ and its torsion $T$ defined by
\begin{equation}\label{eq:R}
    R(X,Y)Z = \nabla_X \nabla_Y Z - \nabla_Y \nabla_X Z - \nabla_{[X,Y]_J} Z,
\end{equation}
and
\begin{equation}\label{eq:T}
    T(X,Y) = \nabla_X Y - \nabla_Y X - [X,Y]_J,
\end{equation}
where $[\cdot , \cdot]_J$ is the Jacobi-Lie bracket on vector fields. We say that the curvature and the torsion is parallel respectively if $\nabla R=0$ and $\nabla T = 0$, that is if
\begin{equation}\label{eq:nablaR=0}
    \nabla_W (R(X,Y)Z) = R(\nabla_W X,Y)Z + R(X,\nabla_W Y) Z + R(X,Y) (\nabla_W Z),
\end{equation}
and
\begin{equation}\label{eq:nablaT=0}
    \nabla_W (T(X,Y)) = T(\nabla_W X, Y) + T(X, \nabla_W Y). 
\end{equation}
When dealing with a connection with parallel curvature and torsion it might be the case that the curvature satisfy the stronger condition $R=0$. The same goes for the torsion, hence we may differentiate between four types of connections with parallel curvature and torsion. This classification shown in the chart below was done by Nomizu in 1954 \cite{Nomizu_54}.
\begin{center}
\vspace{0.2in}{\color{white}pp}
    \begin{tabular}{|l|c|c|} 
    \hline
        & $R=0$ & $\nabla R =0$ \\ 
    \hline
        $T=0$ & flat and torsion free space & locally symmetric space\\ 
    \hline
        $\nabla T=0$ & locally Lie group & locally reductive homogeneous space\\ 
    \hline
    \end{tabular}
    \vspace{0.2in}{\color{white}pp}
\end{center}

\begin{remark}\label{rmk:CanCon}
    Nomizu proved that on any reductive homogeneous space there exists a unique connection satisfying certain properties which he called \textit{the canonical affine connection of the second kind}, see \cite[Theorem 10.2]{Nomizu_54}. Whenever we consider a reductive homogeneous space we will simply call this connection \textit{the canonical connection}. He proceed to show that this connection has parallel curvature and torsion, and that any manifold with such a connection can be locally represented as a reductive homogeneous space \cite[Theorem 18.1]{Nomizu_54}. A symmetric space is a special case of a reductive homogeneous space, and the canonical connection in this case is torsion free. Nomizu showed that any manifold with a connection satisfying $T=0$ and $\nabla R=0$ can be locally represented as a symmetric space \cite[Theorem 17.1]{Nomizu_54}. 
    There are several ways of representing a Lie group G as a reductive homogeneous space. One way is to consider $G \cong G/\{e\}$ where $G$ acts on $M=G/\{e\}$ from the left. In this case the canonical connection is known as the (-)-connection. If we let $G$ act from the right, we get the (+)-connection. In both of these cases the canonical connection is flat and has parallel torsion. The last option is to consider $G\cong G\times G /diag(G)$ where the diagonal $diag(G) \subset G\times G$ is defined as all the elements on the form $(g,g) \in G \times G$. This representation of $G$ is in fact a symmetric space, and the associated canonical connection is torsion free with parallel curvature. This is called the (0)-connection. These connections on Lie groups where studied by Cartan in \cite{Cartan1927}, but the information above can also be found in \cite{Nomizu_54,REGGIANI201326}. Any manifold with a connection satisfying $R=0$ and $\nabla T=0$ can be locally represented as a Lie group with the (-)-connection, and whenever we consider a Lie group we will call this connection \textit{the canonical connection}. In this paper we will always use the canonical connection associated to a manifold when we consider the connection algebra. 
\end{remark}

We conclude this section with the Bianchi identities which are two relations which hold for any manifold with a connection:
\begin{equation}\tag{1B}\label{eq:1Bianchi}
    \sum_{\circlearrowleft (X,Y,Z)} T(T(X,Y),Z) + (\nabla_X T)(Y,Z) - R(X,Y)Z =0, 
\end{equation}
\begin{equation}\tag{2B}\label{eq:2Bianchi}
    \sum_{\circlearrowleft (X,Y,Z)} (\nabla_X R)(Y,Z) - R(X,T(Y,Z)) =0.  
\end{equation}

The notation $\sum_{\circlearrowleft (X,Y,Z)}$ means the sum over all cyclic permutations of the arguments $X,Y,Z$. 

\section{Connection Algebras}\label{sec:1.2}

Consider a vector space $\mathcal{A}$ over the field of real numbers together with a bilinear operator $\rhd$. Then $(\mathcal{A}, \rhd)$ is an algebra. Define the associator $a_{\rhd}(x,y,z) := x \rhd (y \rhd z) - (x \rhd y) \rhd z$ and define the triple-bracket by
\begin{equation}\label{eq:TB}
    [x,y,z] := a_{\rhd}(x,y,z) - a_{\rhd}(y,x,z)
\end{equation}
for $x,y,z \in \mathcal{A}$. 

\begin{remark}
Let $(\frak{X}_M , \nabla)$ be the connection algebra of a smooth manifold $M$. Then, if we let $X\rhd Y := \nabla_X Y$, we get
\begin{equation}\label{eq:TB-RicciId}
    R(X,Y) Z - \nabla_{T(X,Y)} Z = [X,Y,Z].
\end{equation}
This equation is sometimes referred to as Ricci's identity \cite{Ricci_1900} or Ricci's formula \cite{Besse_1987}.
\end{remark}

\begin{note}\label{note:notation}
We give a short overview of some of the notations used in this paper:
\begin{itemize}
    \item When considering a connection algebra we will always let $\nabla = \rhd$, that is $\nabla_X Y = X \rhd Y$. 
    \item We will use capital letters to represent vector fields, $\{ X,Y,Z,W \} \subset \frak{X}_M$, and lower case letters to represent elements of an abstract algebra, $\{x,y,z,w \} \subset \mathcal{A}$.
    \item The notation $[\noarg,\noarg,\noarg]$ will always be used for the triple-bracket defined in \eqref{eq:TB}, while the notation $[\![ \noarg,\noarg,\noarg]\!]$ will be used for other general ternary operators.
    \item The notation $[ \noarg, \noarg]$ is reserved for binary operations that define a Lie algebra, while $[\![ \noarg, \noarg]\!]$ will be used for other binary operations when it is useful to compare to a Lie bracket.
\end{itemize}
\end{note}

\subsection{Pre-Lie Algebra}

\begin{definition}[pre-Lie]
An algebra $\mathcal{A}$ is called a \textit{pre-Lie algebra} if for any $x,y,z \in \mathcal{A}$ it satisfies $[x,y,z] = 0$.
\end{definition}

Consider a locally flat and torsion free space $(M,\nabla)$, that is; $R=0$ and $T=0$. From (\ref{eq:TB-RicciId}) we immediately see that $(\frak{X}_M, \nabla)$ is a pre-Lie algebra. Pre-Lie algebras where first described by Gerstenhaber \cite{Gerstenhaber_63} and by Vinberg \cite{Vinberg}. The relation between these algebras and numerical integration on Euclidean spaces is explored in \cite{B-series_CHV_2010} and \cite{CALAQUE_2011}. The free pre-Lie algebra $\text{pre-Lie}(\mathcal{C})$ over a set $\mathcal{C}$ can be represented by non-planar rooted trees with the grafting product \cite{Chaoton_and_Livernet_01,Dzhumadildeav_and_Löfwall_02}. For instance when $\mathcal{C}$ consist of a single element we get the set of non-planar trees
\[
\text{Tree}(\{\ab \}) = \{ \ab, \aabb, \aaabbb, \aababb, \aaaabbbb, \aaababbb, \aaabbabb = \aabaabbb , \aabababb , \ldots \},
\]
and $\text{pre-Lie}(\mathcal{C}) = \text{span}_{\mathbb{R}}\{ \text{Tree}(\mathcal{C}) \}$. The product can be described as the sum over all terms one can get by grafting the root of the left factor onto a vertex of the right factor. We give an example for $\mathcal{C}= \{ \ab, \AB \}$:
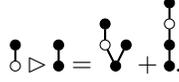
\[
\AabB \rhd \aabb = \aAabBabb + \aaAabBbb.
\]
More information on the formalism of rooted trees will be given in Section \ref{sec:4}. 

\subsection{Post-Lie Algebra}

\begin{definition}[post-Lie]\label{def:post-Lie}
    An algebra $\mathcal{A}$ with product $\rhd$ and a bracket $[\noarg , \noarg]$ is called a \textit{post-Lie algebra} if $(\mathcal{A},[ \noarg, \noarg])$ is a Lie algebra and for all $x,y,z \in \mathcal{A}$ we have
    \begin{enumerate}[(i)]
        \item $x \rhd [y,z] = [x \rhd y, z] + [y, x\rhd z]$,
        \item $[x,y] \rhd z = [x,y,z]$.
    \end{enumerate}
\end{definition}

Let $M$ be locally a Lie group. Then we have a canonical connection satisfying $R=0$ and $\nabla T = 0$ (see Remark \ref{rmk:CanCon}). Define $[X,Y] = -T(X,Y)$, and notice that this bracket is skew symmetric. The first Bianchi identity (\ref{eq:1Bianchi}) gives exactly the Jacobi identity for this bracket, hence $(\frak{X}_M, [\noarg , \noarg ])$ is a Lie algebra. By exchanging the torsion in equation (\ref{eq:nablaT=0}) with this bracket we see that $(\frak{X}_M, [ \noarg , \noarg] , \nabla)$ satisfies the first post-Lie condition. The other follows from inserting $-[X,Y]$ for $T(X,Y)$ in equation (\ref{eq:TB-RicciId}), hence the connection algebra of a local Lie group is a post-Lie algebra.

The theory of post-Lie algebras was developed in conjunction with numerical integration on Lie groups \cite{MKK_03,MK_Wright_08,MK_Lundervold_2013} and has become a major player in this field \cite{KMKL_15,CEFMK19}. Later research has linked post-Lie algebras to geometries beyond Lie groups \cite{MK_Stern_Verdier_20,Gavrilov_paper_2022,grong2023postlie}. The free post-Lie algebra $\text{post-Lie}(\mathcal{C})$ is given by the free Lie algebra over the set of planar rooted trees colored by $\mathcal{C}$ \cite{BV_2007,MK_Lundervold_2013}, here with the left-grafting product. 

\subsection{Lie Admissible Triple Algebra}

\begin{definition}[Lie admissible]
    An algebra $\mathcal{A}$ with product $\rhd$ is \textit{Lie admissible} if it satisfies
    \begin{equation}\label{eq:LieAdm}
        [x,y,z] + [y,z,x] + [z,x,y] = 0
    \end{equation}
    for $x,y,z \in \mathcal{A}$, where $[\noarg , \noarg , \noarg]$ is defined by Equation \eqref{eq:TB}.
\end{definition}

An alternative definition of a Lie admissible algebra is an algebra $(\mathcal{A}, \rhd)$ such that the bracket defined by $[x,y] := x \rhd y - y \rhd x$ makes $(\mathcal{A} ,[\noarg, \noarg])$ into a Lie algebra. A Gröbner basis for the free Lie admissible algebra is given in \cite{Free_Lie-admissible_Grobner}.

\begin{proposition}\label{prop:LieAdm_eq_T=0}
    The connection algebra of a manifold $M$ with an affine connection $\nabla$ is Lie admissible if and only if $T=0$
\end{proposition}

\begin{proof}
    See \cite{MK_Stern_Verdier_20}.
\end{proof} 

If $M$ is any manifold with a connection $\nabla$, then there exist a torsion free connection $\nabla'$ on $M$ such that $\nabla$ and $\nabla'$ give raise to the same geodesics on $M$. This means that all smooth manifolds with an affine connection can be equipped with a Lie admissible connection algebra, which in turn makes Lie admissible connection algebras less interesting for the sort of classification we are interested in, that is: finding one-to-one correspondences between particular types of algebras and particular types of manifolds. The reason we include a discussion about Lie admissible algebras here is because it serves as a nice introduction to Lie admissible triple algebras, which are Lie admissible algebras with additional structure.

\begin{definition}[LAT]\label{def:LAT}
    An algebra $\mathcal{A}$ with product $\rhd$ is a \textit{Lie admissible triple} algebra (LAT) if it satisfies
    \begin{enumerate}[(i)]
        \item $[x,y,z] + [y,z,x] + [z,x,y] = 0$,
        \item $w \rhd [x,y,z] = [w\rhd x, y,z] + [x,w\rhd y, z] + [x,y , w \rhd z]$,
    \end{enumerate}
    for $x,y,z,w \in \mathcal{A}$, where $[\noarg , \noarg , \noarg]$ is defined by Equation \ref{eq:TB}.
\end{definition}

Let $M$ be a locally symmetric space. Then there is a canonical connection satisfying $\nabla R = 0$ and $T=0$. This is a result due to Nomizu, see \cite[Theorem 17.1]{Nomizu_54}. The Ricci identity (\ref{eq:TB-RicciId}) with zero torsion gives $R(X,Y)Z = [X,Y,Z]$. Together with the first Bianchi identity given in Equation \eqref{eq:1Bianchi} we see that the first condition for LAT is satisfied. Moreover, substituting the triple-bracket for the curvature in Equation (\ref{eq:nablaR=0}), we see that $\nabla R = 0$ is equivalent to the second LAT condition. Thus the connection algebra of a locally symmetric space is LAT.

\begin{note}
    The algebraic structure we call LAT has previously been defined under the name G-algebra \cite{GS_00,Sokolov_17}.
\end{note}

\begin{proposition} 
    If $\nabla$ is an affine connection on $M$, then the connection algebra $(\frak{X}_M, \nabla)$ is LAT if and only if $(M, \nabla)$ can be represented as a locally symmetric space with its canonical connection.
\end{proposition}

\begin{proof}
We have already shown that the connection algebra of a symmetric space is LAT. To show the other way we note that if the connection algebra is LAT, then it is also Lie admissible hence $T=0$ by Proposition \ref{prop:LieAdm_eq_T=0}. When the torsion is zero we can identify the curvature and the triple-bracket by Equation (\ref{eq:TB-RicciId}), and then the second condition of LAT is equivalent to $\nabla R = 0$. 

Since $T=0$ and $\nabla R = 0$, $M$ can be represented as a locally symmetric space with its canonical connection by \cite[Theorem 17.1]{Nomizu_54}. 
\end{proof}

Lie admissible triple algebras are closely related to Lie triple systems. A Lie triple system is a vector space $\mathcal{L}$ together with a tri-linear operator $[\![ \noarg , \noarg, \noarg]\!]$ such that
\begin{align}
    [\![x,y,z]\!] &= - [\![y,x,z]\!] \label{LTS1} \\
    [\![x,y,z]\!] + [\![y,z,x]\!] + [\![z,x,y]\!] &= 0 \label{LTS2} \\
    [\![u,v,[\![x,y,z]\!] ]\!] &= [\![ [\![u,v,x]\!],y,z]\!] + [\![x , [\![u,v,y]\!] ,z]\!] + [\![x,y,[\![u,v,z]\!] ]\!] \label{LTS3} 
\end{align}
for all $x,y,z,u,v \in \mathcal{L}$. Lie triple systems relate to symmetric spaces similar to how Lie algebras relate to Lie groups; the tangent space of a symmetric space at some chosen base point can be naturally equipped with a tri-linear operator that satisfies the Lie triple relations \eqref{LTS1}-\eqref{LTS3}.


\begin{proposition}
    If $\mathcal{A}$ is a Lie admissible triple algebra, then we have
    \begin{equation}
        [u,v,[x,y,z] ] = [[u,v,x], y,z] + [x, [u,v,y],z] + [x,y,[u,v,z]] \label{eq:TB_der_prop}
    \end{equation}
    for all $x,y,z,u,v \in \mathcal{A}$, and $(\mathcal{A}, [\noarg, \noarg, \noarg ])$ is a Lie triple system. 
\end{proposition}

\begin{proof}
    The triple bracket satisfies (\ref{LTS1}) by definition, and (\ref{LTS2}) follows from the first defining property of LAT algebras. The last condition (\ref{LTS3}) follows from the second defining property of LAT algebras by a straight forward computation. 
\end{proof}

The relationship between symmetric spaces and Lie triple systems have been investigated in \cite{Loos1} and \cite{Bertram_00}. LAT algebras provide a new perspective to symmetric spaces that was first introduced in \cite{LAT_MKS_2023}.

\section{The Connection Algebra of Reductive Homogeneous Spaces}\label{sec:2}

\subsection{D-algebra}

Consider an algebra $(\mathcal{A} , \rhd )$ and let $T(\mathcal{A}) = \mathbb{I} \oplus \mathcal{A} \oplus (\mathcal{A} \otimes \mathcal{A} ) \oplus \ldots$ be the free tensor algebra over $\mathcal{A}$. For $x,y \in \mathcal{A}$ we write $x \cdot y := x \otimes y$. Let $D(\mathcal{A}) = ( T(\mathcal{A}) , \cdot , \rhd )$ be the free tensor algebra with the product $\rhd$ from $\mathcal{A}$ extended by
\begin{align}
    x \rhd (\omega_1 \cdot \omega_2) &= (x \rhd \omega_1) \cdot \omega_2 + \omega_1 \cdot (x \rhd \omega_2), \label{eq:Dalg1} \\
    (x \cdot \omega_1) \rhd \omega_2 &= a_{\rhd}(x, \omega_1, \omega_2), \label{eq:Dalg2}
\end{align}
for $x \in \mathcal{A}$ and $\omega_1, \omega_2 \in T(\mathcal{A})$. This algebra is called the \textit{D-algebra} of $\mathcal{A}$ and was originally defined in \cite{MK_Wright_08}. Since the tensor product is associative we note that there is a natural Lie bracket $[\noarg , \noarg]$ on $D(\mathcal{A})$ given by $[x,y] = x\cdot y - y \cdot x$. Notice that $[x,y]\rhd z = [x,y,z]$.

\subsection{Post-Lie-Yamaguti Algebra}

\begin{definition}\label{def:InvConAlg}
    An algebra $\mathcal{A}$ with a product $\rhd$ and a bracket $[\![ \noarg , \noarg ]\!]$ is called a \textit{post-Lie-Yamaguti algebra} (PLY) if for all $x,y,z,w \in \mathcal{A}$ it satisfies
    \begin{enumerate}[(i)]
        \item $[\![ \noarg , \noarg ]\!]$ is skew-symmetric, that is $[\![x,x]\!] = 0$,
        \item $w \rhd [\![ x,y ]\!] = [\![ w \rhd x , y ]\!] + [\![ x, w\rhd y ]\!]$,
        \item $w \rhd [x,y,z] = [w \rhd x,y,z] + [x,w \rhd y,z] + [x,y,w \rhd z] + (w \cdot [\![ x,y ]\!] ) \rhd z - [\![ x,y ]\!] \rhd (w \rhd z)$,
        \item $\sum_{ \circlearrowleft (x,y,z)} [x,y,z] = \sum_{ \circlearrowleft (x,y,z)} [\![ [\![ x,y ]\!] , z ]\!] + [\![x,y ]\!] \rhd z$, and
        \item $\sum_{ \circlearrowleft (x,y,z)} [ [\![ x,y ]\!] ,z,w] = \sum_{ \circlearrowleft (x,y,z)} [\![ [\![ x,y]\!] , z]\!] \rhd w$.
    \end{enumerate}
\end{definition}

We remark that the bracket $[\![ \noarg, \noarg ]\!]$ in this definition is generally not a Lie bracket. Also note that the term $(w \cdot [\![ x,y ]\!] ) \rhd z$ in (iii) uses notation from the D-algebra $D(\mathcal{A})$, although it could be written using only elements in $\mathcal{A}$ by \eqref{eq:Dalg2}.

\begin{proposition}
If $\nabla$ is an affine connection on $M$, then the connection algebra $(\mathfrak{X}_M, \nabla)$ is a post-Lie-Yamaguti algebra if and only if $(M, \nabla)$ can be represented as a locally reductive homogeneous space with its canonical connection.
\end{proposition}

\begin{proof}
Note that $M$ can be represented as a locally reductive homogeneous space if and only if it has a connection satisfying $\nabla R = 0$ and $\nabla T =0$. This is a result due to Nomizu, see \cite[Theorem 18.1]{Nomizu_54}. Let $[\![ X, Y ]\!] = -T(X,Y)$. First we assume $M$ to be a locally reductive homogeneous space and show that the connection algebra must be a post-Lie-Yamaguti algebra.
    \begin{enumerate}[(i)]
        \item This follows from the definition of torsion.
        \item This follows immediately from $\nabla T=0$.
        \item Consider $W \rhd [X,Y,Z] - [W \rhd X, Y ,Z] - [X, W\rhd Y, Z] - [X,Y,W\rhd Z]$, and rewrite in terms of curvature and torsion by (\ref{eq:TB-RicciId}). The curvature terms cancel because $\nabla R = 0$. The remaining torsion terms are
        \[
        \begin{split}
            & \, - W \rhd (T(X,Y) \rhd Z) + T(W \rhd X, Y) \rhd Z \\
            & \, + T(X, W\rhd Y) \rhd Z + T(X,Y) \rhd (W \rhd Z) \\
            =& \, -(W \cdot T(X,Y)) \rhd Z + T(X,Y) \rhd (W \rhd Z) \\
            =& \,\, (W \cdot [\![ X,Y ]\!]) \rhd Z - [\![ X,Y ]\!] \rhd (W\rhd Z).
        \end{split}
        \]
        \item This is a consequence of the first Bianchi identity; exchange $R(X,Y)Z$ with $[X,Y,Z] - [\![X,Y]\!] \rhd Z$ and $T(X,Y)$ with $-[\![ X,Y ]\!]$ in Equation (\ref{eq:1Bianchi}). Note that $\nabla T =0$.
        \item This is a consequence of the second Bianchi identity; apply the same exchange to Equation (\ref{eq:2Bianchi}). Note that $\nabla R = 0$.
    \end{enumerate}
    Lastly, we assume that $(\frak{X}_M, \nabla)$ is a post-Lie-Yamaguti algebra. Notice that condition (ii) and (iii) in Definition \ref{def:InvConAlg} are in fact equivalent to $\nabla T=0$ and $\nabla R = 0$ respectively, hence $M$ can be locally represented as a reductive homogeneous space.
\end{proof}

\begin{proposition}
    If $\mathcal{A}$ is a post-Lie-Yamaguti algebra, then $\mathcal{A}$ satisfies
    \begin{equation}\label{PLY:6}
        \begin{split}
            [u,v,[x,y,z]] =& [[u,v,x], y,z] + [x, [u,v,y],z] + [x,y,[u,v,z]] \\
            & + ([u,v]\cdot [\![ x,y]\!] ) \rhd z - [\![x, y]\!] \rhd [u,v,z]
        \end{split}
    \end{equation}
    for all $x,y,z,u,v \in \mathcal{A}$. 
\end{proposition}

\begin{proof}
    Notice that condition (iii) in Definition \ref{def:InvConAlg} is very similar to condition (ii) in Definition \ref{def:LAT}. To prove that any LAT algebra satisfies Equation (\ref{eq:TB_der_prop}) we only need this last property, hence in our computation below we have filtered out all the terms that will sum up to Equation (\ref{eq:TB_der_prop}) and focused on the rest. We get
    \[
    \begin{split}
        [u,v,[x,y,z]] &= u \rhd (v \rhd [x,y,z] ) - (u \rhd v) \rhd [x,y,z] - \{ u \leftrightarrow v \} \\
        &= [[u,v,x], y,z] + [x, [u,v,y],z] + [x,y,[u,v,z]] \\
        & \quad \,\, + \Big( (u \cdot [\![v \rhd  x,y]\!] ) \rhd z + (u \cdot [\![ x,v \rhd y]\!] ) \rhd z \\
        & \quad \,\, + (u \cdot [\![ x,y]\!] ) \rhd (v \rhd z) -[\![v \rhd x,y]\!] \rhd (u\rhd z) \\
        & \quad \,\, -[\![ x,v \rhd y]\!] \rhd (u\rhd z)  -[\![ x,y]\!] \rhd (u\rhd (v\rhd z))  \\
        & \quad \,\, + (u\cdot v \cdot [\![ x,y]\!] ) \rhd z + ((u\rhd v) \cdot [\![ x,y]\!] ) \rhd z \\
        & \quad \,\, + (v \cdot [\![ u\rhd x, y]\!])\rhd z + (v \cdot [\![ x,u\rhd y]\!] ) \rhd z \\
        & \quad \,\, - (u \cdot [\![ x,y]\!] ) \rhd (v \rhd z) - ([\![ u \rhd x,y]\!] ) \rhd (v \rhd z) \\
        & \quad \,\, - ([\![ x,u \rhd y]\!] ) \rhd (v \rhd z) - (( u\rhd v)\cdot [\![x,y ]\!] ) \rhd z \\
        & \quad \,\,+ [\![ x,y]\!] \rhd ((u\rhd v) \rhd z) - \{ u\leftrightarrow v \}  \Big) \\
        &= [[u,v,x], y,z] + [x, [u,v,y],z] + [x,y,[u,v,z]] \\
        & \quad \,\, + \Big( (u\cdot v \cdot [\![x,y]\!] - [\![ x,y]\!] \rhd (u \rhd (v \rhd z)) \\
        & \quad \,\, + [\![x,y]\!] \rhd ( (u \rhd v) \rhd z) - \{ u \leftrightarrow v \} \Big) \\
        &= [[u,v,x], y,z] + [x, [u,v,y],z] + [x,y,[u,v,z]] \\
        & \quad \,\, + ([u,v]\cdot [\![ x,y]\!] ) \rhd z - [\![ x,y]\!] \rhd [u,v,z].
    \end{split}
    \]
    Here $\{ u \leftrightarrow v \}$ means "repeat all the terms inside the same parentheses, but switch the places of $u$ and $v$". Notice that $[x,y,z] = x \rhd (y \rhd z) - (x \rhd y) \rhd z - \{ x \leftrightarrow y \}$. 
\end{proof}

We include a remark to show that PLY is a generalization of both post-Lie and LAT algebras.

\begin{remark}\label{rmrk1}
    A post-Lie algebra is also a PLY algebra: Let $(\mathcal{A}, [\![ \noarg, \noarg]\!] , \rhd )$ be a post-Lie algebra. Then the first post-Lie relation is equal to (ii) in Definition \ref{def:InvConAlg}. The second post-Lie relation implies that (iii) and (v) is satisfied, while (iv) is satisfied by the Jacobi identity of $[\![ \noarg, \noarg ]\!]$ and (i) is satisfied by skew-symmetry of $[\![ \noarg , \noarg ]\!]$.\\
    A LAT algebra $(\mathcal{A}, \rhd )$ with a bracket defined by $[\![ \noarg , \noarg ]\!] = 0$ is a PLY: In this case the relations (i), (ii) and (v) are trivially satisfied, while (iii) and (iv) reduce to exactly the defining relations of LAT. 
\end{remark}

\subsection{Classification of Connection Algebras with Parallel Curvature and Torsion}

In the spirit of Nomizu's classification of manifolds with parallel curvature and torsion, we give a chart that summarizes the previous sections and classifies the different types of connection algebras associated with these manifolds: 

\begin{center}
\vspace{0.2in}{\color{white}pp}
    \begin{tabular}{|l|c|c|} 
    \hline
        & $R=0$ & $\nabla R =0$ \\ 
    \hline
        $T=0$ & pre-Lie algebra & Lie admissible triple algebra \\ 
    \hline
        $\nabla T=0$ & post-Lie algebra & post-Lie-Yamaguti algebra \\ 
    \hline
    \end{tabular}
    \vspace{0.2in}{\color{white}pp}
\end{center}

\section{The Free Algebra}\label{sec:3}

\begin{definition}[Symmetrization map]\label{def:sym_map}
    Let $\mathcal{A}$ be an algebra with product $\rhd$ and let $D(\mathcal{A})$ be the D-algebra of $\mathcal{A}$. Let $x_1 , \ldots , x_n$ be elements of a chosen basis of $\mathcal{A}$. The \textit{symmetrization map} $\frak{s}$ is the linear map defined by 
    \begin{equation}\label{eq:sym_map}
        \begin{split}
            \frak{s} : D(\mathcal{A}) & \longrightarrow D(\mathcal{A}) \\
            x_1 \cdot x_2 \cdots x_n &\longmapsto \frac{1}{n!} \sum_{\sigma \in S_n} x_{\sigma^{-1}(1)} \cdot x_{\sigma^{-1}(2)} \cdots x_{\sigma^{-1}(n)}  
        \end{split}
    \end{equation}
    where $S_n$ is the symmetric group of degree $n$.
\end{definition}

An element of the natural basis of $D(\mathcal{A})$ is called a \textit{word}. If an element $\omega \in D(\mathcal{A})$ is contained in $\mathcal{A}^{\otimes n}$ we will write $l(\omega)= n$ and say that $n$ is the length of $\omega$. We give a few definitions of noteworthy types of elements in $D(\mathcal{A})$:
\begin{itemize}
    \item An element in the image of the symmetrization map $\frak{s}$ is called a \textit{symmetric word}. Note that the empty word is a symmetric word.
    \item An element of the form $\frak{s}(x_1 \cdots x_n) \cdot [y,z]$ is called a \textit{symmetric-bracket word}. Notice the edge case of $n=0$; $[y,z]$ is a symmetric-bracket word.
    \item A symmetric-bracket word $\frak{s}(x_1 \cdots x_n) \cdot [y,z]$ is called \textit{ordered} with respect to a given order $\prec$ on the basis of $\mathcal{A}$ if $y\succ z$ and $x_i \succeq z$ for $1 \leq i \leq n$.
    \item An element $\omega = \omega_1 \cdots \omega_k$ is called an \textit{ordered symmetric-bracket-block word} (or OSBB-word) if $\omega_i$ is an ordered symmetric-bracket word for $i<k$ and $\omega_k$ is a symmetric word. We say that $\omega$ has $k$ blocks. 
\end{itemize}

The elements mentioned in the list above are, in general, linear combinations of words. We will allow this abuse of language since these elements share many of the qualities we commonly associate with words.  

\begin{theorem}\label{th:S}
    Let $\mathcal{A}$ be an algebra with a totally ordered basis. Then any element in $T(\mathcal{A})$ can be uniquely written as a linear combination of OSBB-words.
\end{theorem}

The proof of this theorem can be found in \cite{LAT_MKS_2023}. Another way of stating this theorem is that the OSBB-words of elements from the totally ordered basis of $\mathcal{A}$ is a basis for $T(\mathcal{A})$, and then also for $D(\mathcal{A})$. 

\begin{remark}\label{rem:OSBB}
    Given any word $\eta = x_1 \cdots x_n$ with $x_1, \ldots ,x_n \in \mathcal{A}$, by the theorem above there is a unique way to write $\eta$ as a linear combination of OSBB-words. If $\omega$ is one of the OSBB-words with a non-zero coefficient in this representation of $\eta$, then $l(\omega) = l(\eta)$ and each of the elements $x_i$ will appear exactly one time in the expression of $\omega$. 
\end{remark}

\begin{definition}\label{def:Delta}
    Given a total order $\prec = \prec_B$ on a basis $B$ of $\mathcal{A}$, we define a total order $\prec_{\Delta} = \prec_{\Delta_B}$ on the set of OSBB-words over $B$ in $D(\mathcal{A})$.
    \begin{enumerate}[(i)]
        \item \textbf{Symmetric words:} Let $\omega = \frak{s}(x_1 \cdots x_n)$ and $\eta = \frak{s}(y_1 \cdots y_m)$ be two symmetric words and let the indexing be such that $x_1 \succeq \ldots \succeq x_n$ and $y_1 \succeq \ldots \succeq y_m$. Then $\omega \succ_{\Delta} \eta$ if
        \begin{itemize}
            \item $n>m$ (compare length), or
            \item $l(\omega) = l(\eta)$, $x_i = y_i$ for $i<j$ and $x_j \succ y_j$.
        \end{itemize}
        \item \textbf{Ordered symmetric-bracket words:} Let $\omega = \frak{s}(x_1 \cdots x_n)\cdot[u,v]$ and $\eta = \frak{s}(y_1 \cdots y_m)\cdot[w,z]$ be two ordered symmetric-bracket words and let the indexing be such that $x_1 \succeq \ldots \succeq x_n$ and $y_1 \succeq \ldots \succeq y_m$. Then $\omega \succ_{\Delta} \eta$ if
        \begin{itemize}
            \item $n>m$, or
            \item $n=m$ and $v \succ z$, or
            \item $n=m$, $v=z$ and $u\succ w$, or 
            \item $n=m$, $v=z$, $u=w$, and $\frak{s}(x_1 \cdots x_n) \succeq_{\Delta} \frak{s}(y_1 \cdots y_m)$.
        \end{itemize}
        \item \textbf{Compare symmetric word and ordered symmetric-bracket word:} Let $\omega = \frak{s}(x_1 \cdots x_n)\cdot[u,v]$ be an ordered symmetric-bracket word and $\eta = \frak{s}(y_1 \cdots y_m)$ a symmetric word. 
        Then $\omega \succ_{\Delta} \eta$ if and only if
        \begin{itemize}
            \item $l(\omega) > l(\eta)$.
        \end{itemize}
        \item \textbf{OSBB-words:} let $\omega = \omega_1 \cdots \omega_k$ and $\eta = \eta_1 \cdots \eta_r$ be two OSBB-words. Then $\omega \succ_{\Delta} \eta$ if
        \begin{itemize}
            \item $l(\omega)>l(\eta)$, or 
            \item $l(\omega)=l(\eta)$, $\omega_i = \eta_i$ for $i<j$ and $\omega_j \succ_{\Delta} \eta_j$.
        \end{itemize}
    \end{enumerate}
\end{definition}

We must remark that there is nothing special about this order as far as the author knows. It is included to make the constructions in later sections concrete. 

\subsection{Trees}

Let $\Alg(\mathcal{C} )$ be the free algebra over $\mathbb{R}$ generated by a set $\mathcal{C}$ and a binary operator $\rhd$. The set $\OT(\mathcal{C})$ of \textit{planar rooted trees} with nodes colored by $\mathcal{C}$ is a basis for $\Alg(\mathcal{C})$, see for instance \cite{Al-Kaabi_14}. We give the first few planar rooted trees with a single color:
\[
\OT(\{\ab\}) = \{ \ab , \aabb, \aaabbb, \aababb , \aaaabbbb, \aaababbb , \aaabbabb , \aabaabbb, \aabababb , \ldots \} .
\]
Each planar rooted tree can be uniquely described by giving its \textit{branches} in order together with its \textit{root}. The root of a planar rooted tree is an element from the set $\mathcal{C}$ and the branches are themselves planar rooted trees. For example the tree $\AabaabbB$ has the root $\AB$ and the two branches $\ab$ and $\aabb$ in that order, and we can write 
\[
\AabaabbB = t(\AB; \ab, \aabb).
\]
In general if $\tau \in \OT(\mathcal{C})$ with a root $c\in \mathcal{C}$ and branches $\tau_1 , \ldots , \tau_n \in \OT(\mathcal{C})$ in that order, we can write $\tau = t( c ; \tau_1 , \ldots , \tau_n)$. Let $\OF(\mathcal{C})$ denote the set of \textit{ordered forests} of planar rooted trees with nodes colored by $\mathcal{C}$: 
\[
\OF(\{\ab \}) = \{ \mathbb{I} , \ab , \aabb , \ab\, \ab , \aaabbb, \aababb, \aabb \, \ab , \ab \, \aabb , \ab \, \ab \, \ab , \ldots \}.
\]
First we note that for $\tau = t( c ; \tau_1 , \ldots , \tau_n) \in \OT(\mathcal{C})$ the collection of branches can be thought of as an ordered forest $\tau_1 \tau_2 \ldots \tau_n \in \OF(\mathcal{C})$. Secondly we note that there is a trivial injective map from $\OF(\mathcal{C})$ to $D(\Alg(\mathcal{C}))$ given by $\tau_1 \tau_2 \ldots \tau_n \mapsto \tau_1 \cdot \tau_2 \cdots \tau_n$. For simplicity we will consider $\OF(\mathcal{C})$ as a subset of $D(\Alg(\mathcal{C}))$. Lastly we note the following relation
\begin{equation}\label{eq:tree_from_forest}
    \tau = t( c ; \tau_1 , \ldots , \tau_n) = (\tau_1 \cdots \tau_n) \rhd c
\end{equation}
which follows from Equation (\ref{eq:Dalg2}). 

For a tree $\tau \in \OT(\mathcal{C})$ we let $|\tau|_V$ denote the number of vertices in $\tau$, e.g. $|\aabababb|_V=4$. Extend this definition to forests by $|\tau_1 \cdots \tau_n|_V= |\tau_1|_V + \ldots + |\tau_N|_V$.

\subsection{Two Bases of the Free Algebra}

Whenever needed, we will assume that $\mathcal{C}$ is a totally ordered set with the order given as $\prec_{\mathcal{C}}$. By Equation (\ref{eq:tree_from_forest}) we can give an iterative definition of the set $\OT(\mathcal{C})$:
\begin{enumerate}[(i)]
    \item if $c \in \mathcal{C}$ then $c \in \OT(\mathcal{C})$, and
    \item if $x_1, \ldots , x_n \in \OT(\mathcal{C})$ and $c \in \mathcal{C}$, then $(x_1 \cdots x_n) \rhd c \in \OT(\mathcal{C})$. 
\end{enumerate}

We want to construct another basis of $\Alg(\mathcal{C})$ by using Theorem \ref{th:S}. The way to generate new ordered rooted trees is to select any number of ordered rooted trees already constructed (that is any number of elements already in the basis), choose a root (that is an element from $\mathcal{C}$), and then choose an order in which to connect the family of trees to the root (corresponding to choosing a permutation). We can follow the same steps, but instead of working with all permutations we use all the OSBB-words, and by Theorem \ref{th:S} the result will be another basis of $\Alg(\mathcal{C})$. We will call this basis $S$. One challenge when constructing $S$ is that to create an OSBB-word of elements in $S$ we need an order on $S$, hence the order must be expanded simultaneously with the construction of the set. We refer to \cite{LAT_MKS_2023} for more details.

    Let $S \subset \Alg(\mathcal{C})$ be a set with a total order $\prec = \prec_S$ such that
    \begin{enumerate}[(i)]
        \item If $c \in \mathcal{C}$ then $c \in S$, and $c_1 \succ c_2$ if $c_1 \succ_{\mathcal{C}} c_2$.
        \item If $\omega \in D(\Alg(\mathcal{C}))$ is an OSBB-word of elements from $S$ and $c \in \mathcal{C}$, then $\omega \rhd c$ is in $S$, and $\omega \rhd c \succ \eta \rhd d$ if
        \begin{itemize}
            \item $|\omega|_V > |\eta|_V$, or
            \item $|\omega|_V = |\eta|_V$ and $\omega \succ_{\Delta_S} \eta$, or
            \item $\omega = \eta$ and $c \succ_{\mathcal{C}} d$.
        \end{itemize}
    \end{enumerate}

We list the elements of $S \subset \Alg(\{\ab\})$ with up to five vertices in increasing order:

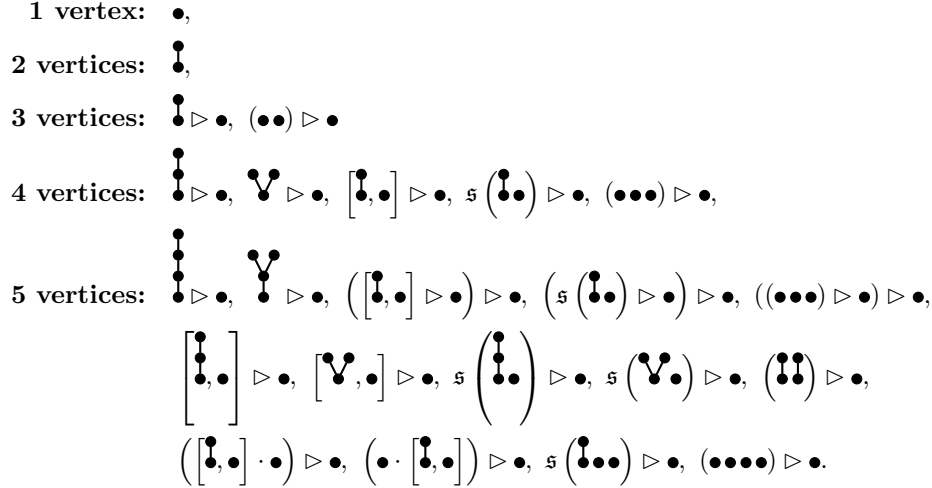
\[
\begin{split}
    \textbf{1 vertex:}& \quad \ab, \\
    \textbf{2 vertices:}& \quad \aabb, \\
    \textbf{3 vertices:}& \quad \aabb  \rhd \ab, \,\, (\ab \, \ab) \rhd \ab  \\
    \textbf{4 vertices:}& \quad \aaabbb  \rhd \ab , \,\, \aababb  \rhd \ab, \,\, \left[\aabb , \ab \right] \rhd \ab , \,\, \frak{s}\left(\aabb \, \ab \right) \rhd \ab, \,\, (\ab \, \ab \, \ab ) \rhd \ab, \\
    \textbf{5 vertices:}& \quad \aaaabbbb \rhd \ab, \,\, \aaababbb  \rhd \ab, \,\, \left( \left[\aabb , \ab\right] \rhd \ab \right) \rhd \ab, \,\, \left( \frak{s}\left(\aabb \, \ab\right) \rhd \ab\right) \rhd \ab, \,\, \left( (\ab \, \ab \, \ab ) \rhd \ab \right) \rhd \ab, \\
    & \quad \left[ \aaabbb, \ab \right] \rhd \ab , \,\, \left[\aababb , \ab \right] \rhd \ab , \,\, \frak{s}\left( \aaabbb \, \ab \right) \rhd \ab , \,\, \frak{s} \left( \aababb \ab \right) \rhd \ab , \,\, \left( \aabb \, \aabb \right) \rhd \ab, \\
    & \quad \left( \left[\aabb , \ab \right] \cdot \ab \right) \rhd \ab, \,\, \left( \ab \cdot \left[\aabb, \ab \right] \right) \rhd \ab, \,\, \frak{s}\left( \aabb \, \ab \, \ab \right) \rhd \ab, \,\, (\ab \, \ab \, \ab \, \ab ) \rhd \ab.
\end{split}
\]

\section{The Free Algebra Over Two Operators}\label{sec:4}

Let $A := \Alg(\mathcal{C} ; [\![ \noarg, \noarg ]\!] , \rhd )$ be the free algebra generated by the set $\mathcal{C}$ and the two binary operators $[\![ \noarg , \noarg ]\!]$ and $\rhd$. For clarity we remark that the free algebra $\Alg(\mathcal{C}) = \Alg(\mathcal{C}; \rhd)$, but we usually omit the operator from the notation in this case. We let $D(A)$ be the free tensor algebra over $A$ with the triangle product extended as in Equation (\ref{eq:Dalg1}) and (\ref{eq:Dalg2}). Note that we do not extend the bracket $[\![ \noarg , \noarg]\!]$ to $D(A)$. 
The first elements in the case $\mathcal{C}= \{\ab\}$ are
\[
\ab, \,\aabb, \, [\![ \ab, \ab ]\!], \, \aaabbb, \, \aababb, \, [\![ \ab, \ab ]\!] \rhd \ab ,\, \ab \rhd [\![ \ab , \ab ]\!],\, [\![ \aabb, \ab ]\!],\, [\![ \ab , \aabb ]\!],\, [\![ [\![ \ab, \ab ]\!], \ab ]\!] ,\, [\![ \ab , [\![ \ab , \ab ]\!] ]\!] ,\, \ldots 
\]
We might still think of the elements in the basis of $A$ as planar rooted trees and nested brackets of such, but with vertices that can be "colored" by either a bracket or an element from $\mathcal{C}$. For example, if we let $[\![ \ab , \ab ]\!] = \AB$, then $[\![ \ab, \ab ]\!] \rhd \ab = \aABb$ and $\ab \rhd [\![ \ab , \ab ]\!] = \AabB$. By this intuition, the general element in this basis of $A$ can be seen as a planar rooted tree $t$ with a function that map each vertex of $t$ to an element in $\mathcal{C} \cup \text{Brack}(A)$ where $\text{Brack}(A)$ is all elements on the form $[\![ x,y]\!]$ where $x$ and $y$ are in the basis of $A$. This will generate the basis $\hat{B}$ described below. We generalize the definition of $|\noarg|_V$ to these trees by letting $|[\![ x, y]\!] |_V = |x|_V + |y|_V$. Note that if we interpret $\aABb$ as a tree with the the vertex $\AB = [\![ \ab, \ab ]\!]$, this must be considered a weighted vertex, such that $|\AB|_V=|[\![ \ab, \ab ]\!]|_V=2$ and $|\aABb|_V = 3$.
\begin{remark}
    The number $\alpha_n$ of trees in $\OT(\mathcal{C})$ with $n$ vertices are given by the Catalan numbers multiplied by $|\mathcal{C}|^n$
    \[
    \alpha_n = \frac{|\mathcal{C}|^n}{n} \binom{2n-2}{n-1},
    \]
    where $|\mathcal{C}|$ is the number of elements in $\mathcal{C}$. This does of course correspond to the number of ways to select $n$ elements from $\mathcal{C}$ in a specific order, and then put parentheses that determine in which order to multiply by the product $\rhd$. When considering $A$, the only difference is that when putting a parenthesis, we must choose which operator to apply, between $\rhd$ and $[\![ \noarg, \noarg ]\!]$, hence the number $\beta_n$ of trees in $A$ with $n$ vertices is given by
    \[
    \beta_n =  \frac{2^{n-1}|\mathcal{C}|^n}{n} \binom{2n-2}{n-1}.
    \]
\end{remark}

We generalize the basis $\OT(\mathcal{C})$ of $\Alg(\mathcal{C})$ to a basis $\hat{B}$ of $A$:
\begin{enumerate}[(i)]
    \item if $c \in \mathcal{C}$, then $c \in \hat{B}$,
    \item if $x,y \in \hat{B}$, then $[\![ x,y ]\!] \in \hat{B}$, and 
    \item if $x_1 , \ldots ,x_n, r \in \hat{B}$ with $r \in \mathcal{C}$ or $r= [\![y,z]\!]$, then $(x_1 \cdots x_n ) \rhd r \in \hat{B}$.
\end{enumerate}

We can also generalize the basis $S$ of $\Alg(\mathcal{C})$ to a basis $\hat{S}$ of $A$:
\begin{enumerate}[(i)]
    \item if $c \in \mathcal{C}$, then $c \in \hat{S}$,
    \item if $x,y \in \hat{S}$, then $[\![ x,y ]\!] \in \hat{S}$, and 
    \item if $\omega \in D(A)$ is an OSBB-word of elements from $\hat{S}$ and $r \in \hat{S}$ with $r \in \mathcal{C}$ or $r= [\![y,z]\!]$, then $\omega \rhd r \in \hat{S}$,
\end{enumerate}
where we use the following order $\prec = \prec_{\hat{S}}$ on $\hat{S}$ when determining the OSBB-words:
\begin{itemize}
    \item If $|x|_V > |y|_V$, then $x \succ y$.
    \item If $|x|_V = |y|_V = 1$ and $x \succ_{\mathcal{C}} y$, then $x \succ y$. 
    \item If $|x|_V = |y|_V$, $x = [\![ u,v ]\!]$, $y= [\![ w,z]\!]$, then $x \succ y$ if
    \begin{itemize}
        \item $u \succ w$, or
        \item $u=w$ and $v \succ z$.
    \end{itemize}
    \item If $|x|_V = |y|_V$ and $x = \omega \rhd r$, $y= \eta \rhd s$, then $x \succ y$ if
    \begin{itemize}
        \item $|\omega|_V > |\eta|_V$, or
        \item $|\omega|_V = |\eta|_V$ and $\omega \succ_{\Delta} \eta$, or
        \item $\omega = \eta$ and $r \succ s$.
    \end{itemize} 
\end{itemize}

We list all the elements $x \in \hat{S}$ with $|x|_V\leq 4$ when $\mathcal{C} = \{\ab \}$ in increasing order where we let $\AB := [\![ \ab ,\ab ]\!]$ and simplify the expression into a planar rooted tree colored by $\{ \ab, \AB \}$ whenever possible:
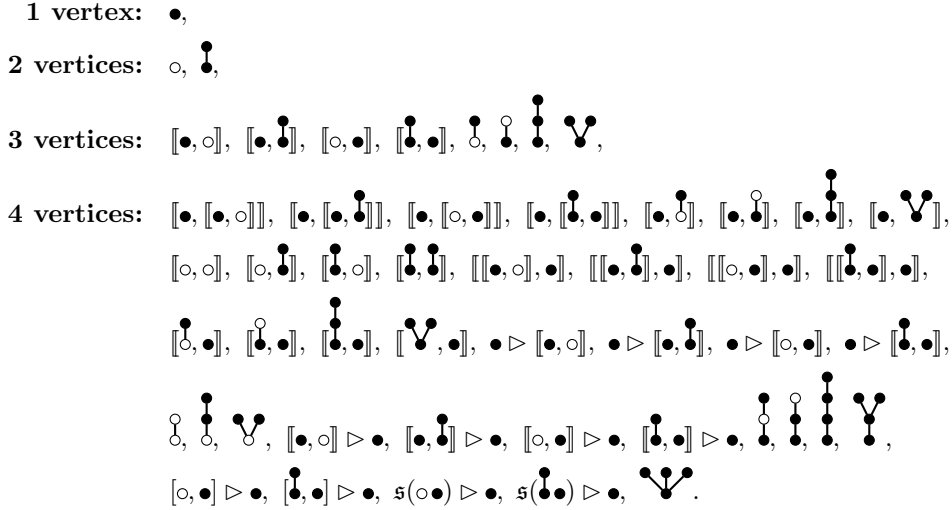
\[
\begin{split}
    \textbf{1 vertex:}& \quad \ab, \\
    \textbf{2 vertices:}& \quad \AB, \,\, \aabb, \\
    \textbf{3 vertices:}& \quad [\![ \ab , \AB]\!], \,\, [\![ \ab , \aabb ]\!], \,\, [\![\AB , \ab ]\!], \,\, [\![ \aabb , \ab ]\!], \,\, \AabB, \,\, \aABb, \,\, \aaabbb, \,\, \aababb, \\
    \textbf{4 vertices:}& \quad [\![ \ab ,[\![ \ab , \AB]\!] ]\!], \,\, [\![ \ab , [\![ \ab , \aabb ]\!] ]\!], \,\, [\![ \ab , [\![\AB , \ab ]\!] ]\!], \,\, [\![ \ab , [\![ \aabb , \ab ]\!] ]\!], \,\, [\![ \ab , \AabB ]\!], \,\, [\![ \ab ,\aABb ]\!], \,\, [\![ \ab,  \aaabbb ]\!], \,\, [\![ \ab , \aababb]\!],\\
    & \quad  [\![ \AB , \AB]\!], \,\, [\![ \AB , \aabb]\!], \,\, [\![ \aabb, \AB ]\!], \,\, [\![ \aabb, \aabb]\!], \,\, [\![ [\![ \ab , \AB]\!] , \ab ]\!], \,\, [\![ [\![ \ab , \aabb ]\!] , \ab ]\!], \,\, [\![ [\![\AB , \ab ]\!] , \ab ]\!], \,\, [\![ [\![ \aabb , \ab ]\!] , \ab ]\!], \\
    & \quad [\![  \AabB , \ab ]\!], \,\, [\![ \aABb , \ab ]\!], \,\, [\![ \aaabbb , \ab ]\!], \,\, [\![ \aababb , \ab ]\!], \,\, \ab\rhd [\![ \ab , \AB]\!], \,\, \ab \rhd [\![ \ab, \aabb]\!], \,\, \ab \rhd [\![ \AB, \ab]\!], \,\, \ab \rhd [\![ \aabb, \ab ]\!], \\
    & \quad \AABB , \,\, \AaabbB, \,\, \AababB, \,\, [\![\ab , \AB]\!] \rhd \ab, \,\, [\![ \ab , \aabb ]\!] \rhd \ab, \,\, [\![ \AB, \ab ]\!] \rhd \ab, \,\, [\![ \aabb, \ab ]\!]  \rhd \ab, \,\, \aAabBb, \,\, \aaABbb, \,\, \aaaabbbb, \,\, \aaababbb, \\
    & \quad   [\AB, \ab]\rhd \ab, \,\, [\aabb, \ab] \rhd \ab, \,\,  \frak{s}(\AB \, \ab) \rhd \ab, \,\, \frak{s}(\aabb \, \ab) \rhd \ab, \,\, \aabababb.
\end{split}
\]

\subsection{Properties of the Basis $\hat{S}$}

While the binary operator $[\![ \noarg , \noarg ]\!]$ is kept visible in its most natural form in the elements of $\hat{S}$, the binary operator $\rhd$ is rewritten in the formalism of planar rooted trees. Therefore the element $x \rhd y$ is generally not an element of $\hat{S}$ even if $x$ and $y$ both are in $\hat{S}$. It will nevertheless be important for later proofs to have some knowledge of the element $x \rhd y$ without having to compute its representative in the basis $\hat{S}$. We note that any element in $\hat{S}$ is of the form $\omega \rhd r$ where $\omega$ is an OSBB-word of elements from $\hat{S}$ (potentially the empty word) and $r$ is either in $\mathcal{C}$ or on the form $[\![ u,v]\!]$ with $u,v \in \hat{S}$. 

For any $x \in \hat{S}$, we define $b(x)$ by
\begin{itemize}
    \item $b(c)=0$ for $c \in \mathcal{C}$,
    \item $b([\![ x,y ]\!])= 1 + b(x) + b(y)$,
    \item $b(\omega \rhd r) = b(\omega) + b(r)$ where $b(\omega) = b(\omega_1) + \ldots + b(\omega_k)$ for where $\omega_i$ are the blocks of $\omega$ and $b(\omega_i) = b(x_1) + \ldots + b(x_n) + b(y) + b(z)$ for $\omega_i = \frak{s}(x_1 \cdots x_n)\cdot[y,z]$. 
\end{itemize}
In simple words $b(x)$ counts the number of times the bracket $[\![ \noarg , \noarg ]\!]$ was used to generate $x$ from $\mathcal{C}$. A simple but important fact is that $b(x)$ is finite for any finitely generated element $x$, and in particular $b(x) < |x|_V$.

\begin{lemma}\label{lemma:S}
    Let $x,z \in \hat{S}$ with $z = \omega \rhd r$. Let $x\rhd z = \sum_i a_i (\eta_i \rhd s_i)$ where $a_i \in \mathbb{R}$ and $\eta_i \rhd s_i$ are in $\hat{S}$. Then
    \begin{enumerate}[(i)]
        \item $s_i = r$, 
        \item $\sup_i \{ l(\eta_i) \} = l(\omega) +1$, 
        \item $|\eta_i|_V = |x|_V + |\omega|_V$, and 
        \item $b(\eta_i \rhd r) = b(x) + b(z)$. 
    \end{enumerate}
\end{lemma}

\begin{proof}
    See that
    \[
    x \rhd z = x \rhd (\omega \rhd r) = (x \cdot \omega) \rhd r + (x \rhd \omega) \rhd r.
    \]
    Notice that if we rewrite $x \cdot \omega + x \rhd \omega$ as a linear combination of OSBB-words $\sum_{i \in I} \eta_i$ we get $(x \cdot \omega) \rhd r + (x \rhd \omega) \rhd r = \sum_{i\in I} \eta_i \rhd r$, and (i) follows since $\eta_i \rhd r$ is in $\hat{S}$ by definition. Notice that $l(x \rhd \omega) = l(\omega)$ by Equation (\ref{eq:Dalg1}) in the definition of a D-algebra. By Remark \ref{rem:OSBB} $\eta_i$ must must have the same length as either $x \cdot \omega$ or $x \rhd \omega$, hence $\sup_i \{l(\eta_i) \} = l(x \cdot \omega ) = l(\omega) +1$ which proves (ii). By the same remark we know that each $\eta_i$ is composed of exactly all the elements that are present in $x$ and $\omega$, hence the number of vertices is preserved, which proves (iii). Lastly we note from the definition of $b(\noarg)$ it is clear that $b(( x \cdot \omega) \rhd r) = b(\eta_i \rhd r)$ for $i \in I_0 \subset I$ with $( x \cdot \omega) \rhd r = \sum_{i \in I_0} \eta_i \rhd r$. For the other term, we use equation (\ref{eq:Dalg1}) and get
    \[
    \begin{split}
        (x \rhd \omega) \rhd r &= \sum_j (\omega_1 \cdots (x \rhd \omega_j) \cdots \omega_k) \rhd r \\
        &= \sum_j \Big( \sum_s (\omega_1 \cdots \frak{s}(y_1 \cdots (x\rhd y_s) \cdots y_m)\cdot [u,v] \cdots \omega_k) \rhd r  \\
        & \quad \quad \quad \quad + (\omega_1 \cdots \frak{s}(y_1 \cdots  y_m)\cdot [x \rhd u,v] \cdots \omega_k) \rhd r \\
        & \quad \quad \quad \quad \quad  + (\omega_1 \cdots \frak{s}(y_1 \cdots  y_m)\cdot [u,x \rhd v] \cdots \omega_k) \rhd r \Big) \\
    \end{split}
    \]
    where $\omega_j = \frak{s}(y_1 \cdots y_m) \cdot [u,v]$. Notice that $b(c_1 \rhd c_2) = b(c_1) + b(c_2) = 0$ for $c_1, c_2 \in \mathcal{C}$. In other words (iv) is true for $|x \rhd z|_V = 2$. Since $|x \rhd y_j|_V$, $|x \rhd u|_V$ and $|x \rhd v|_V$ are all smaller than $|x\rhd (\omega \rhd r)|_V$, the statement follows by induction on $|\noarg |_V$.
\end{proof}

\begin{cor}\label{cor:S}
    Let $x,y,z \in \hat{S}$ where $z = \omega \rhd r$, and let $\eta$ be an OSBB-word over $\hat{S}$. If we rewrite $[x,y,z]$ and $\eta \rhd z$ in the basis $\hat{S}$, then
    \begin{enumerate}[(i)]
        \item all the terms will be on the form $\nu \rhd r$,
        \item for $[x,y,z]$ we get $\sup \{ l(\nu) \} = l(\omega) +2$ and for $\eta \rhd z$ we get $\sup_i \{ l(\nu) \} = l(\omega) + l(\eta)$,
        \item the number of vertices $|\noarg|_V$ will be preserved, and
        \item the number of brackets $b(\noarg)$ will be preserved.
    \end{enumerate}
\end{cor}

\begin{proof}
    The proof follows from rewriting $[x,y,z]$ and $\eta \rhd z$ with only the product $\rhd$ and elements from $\hat{S}$ by Equation (\ref{eq:TB}) and Equation (\ref{eq:Dalg2}) respectively, then everything follows from Lemma \ref{lemma:S}.
\end{proof}

\begin{lemma}\label{lemma:TB_largest_term}
    Let $x_1, \ldots , x_n,y,z,w \in \hat{S}$ with $w = \omega \rhd r$, then
    \[
    \frak{s}(x_1 \cdots x_n) \rhd [y,z,w] = (\frak{s}(x_1 \cdots x_n) \cdot [y,z] \cdot \omega) \rhd r + \sum_i \nu_i \rhd r
    \]
    where $l(\nu_i) < l(\frak{s}(x_1 \cdots x_n) \cdot [y,z] \cdot \omega)$ for all $i$. 
\end{lemma}

\begin{proof}
    If $n=0$ we notice that $[y,z,\omega \rhd r] = ([y,z]\cdot \omega ) \rhd r + [y,z,\omega] \rhd r$ where $l([y,z,\omega]) = l(\omega) < l([y,z] \cdot \omega)$ by Equation (\ref{eq:Dalg1}). 
     Assume the lemma is true for $n<k$ and consider $\frak{s}(x_1 \cdots x_k) \rhd [y,z,w]$. We have
    \[
    \begin{split}
        \frak{s}(x_1 \cdots x_n) \rhd [y,z,w] &= \sum_{i=1}^k \Big( x_i \cdot \frak{s}(x_1 \cdots \hat{x}_i \cdots x_k) \Big) \rhd [y,z,w] \\
        &= \sum_{i=1}^k \Big( x_i \rhd (\frak{s}(x_1 \cdots \hat{x}_i \cdots x_k)  \rhd [y,z,w] ) \\
        & \quad \quad \quad \quad - ( x_i \rhd \frak{s}(x_1 \cdots \hat{x}_i \cdots x_k) ) \rhd [y,z,w] \Big). \\
    \end{split}
    \]
     By assumption and Equation (\ref{eq:Dalg1}) we have $\Big( x_i \rhd \frak{s}(x_1 \cdots \hat{x}_i \cdots x_k) \Big) \rhd [y,z,w] = \eta^i \rhd r + \sum_j \nu_j^i \rhd r$ where $l(\nu_j^i) < l(\eta^i) = l(\omega) + k+1$. Also by assumption, we have $\frak{s}(x_1 \cdots \hat{x}_i \cdots x_k)  \rhd [y,z,w] = (\frak{s}(x_1 \cdots \hat{x}_i \cdots x_k) \cdot [y,z] \cdot \omega ) \rhd r + \sum_j \mu^i_j \rhd r$ with $l(\mu^i_j) < l(\omega) + k +1$. The lemma follows from
     \[
     \sum_{i=1}^k x_i \cdot \frak{s}(x_1 \cdots \hat{x}_i \cdots x_n) \cdot [y,z] \cdot \omega = \frak{s}(x_1 \cdots x_n) \cdot [y,z] \cdot \omega
     \] 
     since $l(\frak{s}(x_1 \cdots x_n) \cdot [y,z] \cdot \omega) = l(\omega) + k + 2$ which is larger than the other terms. 
\end{proof}

\subsection{A Basis Generated from the Triple-Bracket}

\begin{definition}\label{def:T_basis}
    Let $\T \subset A$ be defined by
    \begin{enumerate}[(i)]
        \item if $c \in \mathcal{C}$, then $c \in \T$,
        \item if $u,v \in \T$, then $[\![ u,v]\!] \in \T$,
        \item if $x_1,\ldots x_n, r \in \T$ with $r \in \mathcal{C}$ or $r = [\![ u,v]\!]$, then $\frak{s}(x_1 \cdots x_n) \rhd r \in \T$,
        \item if $x_1 , \ldots , x_n , y,z,w \in \T$ with
        \begin{itemize}
            \item $y \succ z$, and
            \item $ x_i \succeq z$ for $1 \leq i \leq n$
        \end{itemize}
        then $\frak{s}(x_1 \cdots x_n) \rhd [y,z,w] \in \T$.
    \end{enumerate}
\end{definition}

Notice that in this definition, we do allow $n=0$, hence $[y,z,w] \in\T$ if $y \succ z$. 

\begin{proposition}\label{prop:T_basis}
    The set $\T$ is a basis of $A$. Let
    \[
    \begin{split}
        \phi: A & \longmapsto A \\
        c &\mapsto c \\
        [\![ u,v]\!] & \longmapsto [\![ \phi(u), \phi(v)]\!] \\
        \omega \rhd r & \longmapsto \frak{s}(\phi(x_1) \cdots \phi(x_n)) \rhd [\phi(y),\phi(z), \phi(\Bar{\omega} \rhd r)]
    \end{split}
    \]
    where $ \omega = \omega_1 \cdots \omega_k$ is an OSBB-word with $\omega_1 = \frak{s}(x_1 \cdots x_n)\cdot [y,z]$ and $\Bar{\omega}= \omega_2 \cdots \omega_k$. Then $\phi$ is an automorphism and $\phi|_{\hat{S}}$ is a bijection between $\hat{S}$ and $\T$.
\end{proposition}

\begin{proof}
    The map $\phi$ is explicitly defined on the basis $\hat{S}$ and it is well defined.
    First, notice that the map $\phi$ preserves the number of vertices $| \noarg |_V$. Let $V_i$ be the graded component of $A$ consisting of all elements with exactly $i$ vertices; $V_i := \{ x \in A \, | \, |x|_V = i \}$. Clearly $\phi|_{V_1}$ is an automorphism of $V_1$. Assume $\phi|_{V_i}$ is an automorphism of $V_i$ for $i<k$ and let $x \in V_k$.
    
    \textbf{Claim:} $\phi(x) = x + \sum_i y_i$ where $x, y_i \in \hat{S}$ and $x \succ_{\hat{S}} y_i$.
    
    Notice that this claim is true for $z \in V_1$, and assume that it is true for $z \in V_i$ with $i<k$. Also assume that it is true for all $z \in V_k$ with $z \prec_{\hat{S}} x$. There are two cases to consider:
    \begin{itemize}
        \item If $x = [\![ u,v]\!]$ we have 
        \[
        \begin{split}
            \phi(x) &= [\![ \phi(u) , \phi(v)]\!] \\
            &= [\![u + \sum_i y_i, v + \sum_j z_j]\!] \quad \text{ (by assumtion)} \\
            &= x + \sum_{i,j} w_{i,j} 
        \end{split}
        \]
        where $w_{i,j} = [\![ y_i,z_j]\!]$ for $i,j>0$, $w_{0,0} = 0$, $w_{0,j} = [\![u,z_j]\!]$ and $w_{i,0} = [\![y_i, v]\!]$. It is clear by assumption that $x \succ_{\hat{S}} w_{i,j}$.
        \item If $x = \omega \rhd r = (\frak{s}(x_1 \cdots x_n) [y,z] \Bar{\omega}) \rhd r$ we have
        \[
        \begin{split}
            \phi(x) &= \frak{s}(\phi(x_1) \cdots \phi(x_n)) \rhd [\phi(y),\phi(z), \phi( \Bar{\omega} \rhd r)] \\
            &= \frak{s}(x_1 \cdots x_n) \rhd [y,z,  \Bar{\omega} \rhd r] + \sum_{i_1 , \ldots , i_n, j,s,t} \frak{s}(x_1^{i_1} \cdots x_n^{i_n}) \rhd [y_j,z_s,  \nu_t \rhd r_t] 
        \end{split}
        \]
        by assumption and linearity, where $x_l \succeq_{\hat{S}} x_l^{i_l}$, $y \succeq_{\hat{S}} y_j$, $z \succeq_{\hat{S}} z_s$ and $\Bar{\omega} \rhd r \succeq_{\hat{S}} \nu_t \rhd r_t$. By Lemma \ref{lemma:TB_largest_term} we have
        \[
        \begin{split}
            \frak{s}(x_1 \cdots x_n) \rhd [y,z,  \Bar{\omega} \rhd r] &= x + \sum_m \eta_m \rhd r \\
            \frak{s}(x_1^{i_1} \cdots x_n^{i_n}) \rhd [y_j,z_s,  \nu_t \rhd r_t] &= (\frak{s}(x_1^{i_1} \cdots x_n^{i_n}) \cdot [y_j, z_s] \cdot \nu_t) \rhd r_t + \sum_l \mu_l \rhd r_t
        \end{split}
        \]
        with $x \succ_{\hat{S}} \eta_m \rhd r$ and $(\frak{s}(x_1^{i_1} \cdots x_n^{i_n}) \cdot [y_j, z_s] \cdot \nu_t) \rhd r_t \succ_{\hat{S}} \mu_l \rhd r_t$. Since all the factors are pairwise bigger in $x$ than in $(\frak{s}(x_1^{i_1} \cdots x_n^{i_n}) \cdot [y_j, z_s] \cdot \nu_t) \rhd r_t$ we must have $x \succ_{\hat{S}} (\frak{s}(x_1^{i_1} \cdots x_n^{i_n}) \cdot [y_j, z_s] \cdot \nu_t) \rhd r_t$, even after rewriting $(\frak{s}(x_1^{i_1} \cdots x_n^{i_n}) \cdot [y_j, z_s] \cdot \nu_t) \rhd r_t$ in $\hat{S}$, see \cite[Lemma 3.11]{LAT_MKS_2023}.
    \end{itemize}
    This proves the claim, and from the claim, it follows that $\phi$ maps each graded component $V_i$ to itself as an upper triangular matrix with respect to the basis $\hat{S}$ where all the entries on the diagonal are $1$, hence the map has an inverse and must be an automorphism. 
    
    To prove that $\phi|_{\hat{S}}$ is a bijection onto $\T$, we again use induction on $|\noarg|_V$. This is clear for $V_1$, and assume it is the case for $y \in V_i$ with $i<k$. Let $x \in V_k$. If $x=[\![ u,v]\!]$ we have $\phi(x) = [\![ \phi(u), \phi(v)]\!]$ and by assumption $\phi(u)$ and $\phi(v)$ are both in $\T$. Then, by the definition of $\T$, $\phi(x) \in \T$. If $x = \omega \rhd r$ we have $\phi(x) =\frak{s}(\phi(x_1) \cdots \phi(x_n)) \rhd [\phi(y),\phi(z), \phi(\Bar{\omega} \rhd r)] $, and since all the arguments inside this term have fewer vertices than $x$ they are in $\T$ by assumption, hence $\phi(x)$ is in $\T$ by the definition of $\T$. This proves $\phi(\hat{S}) \subset \T$. Now let $x \in \T$. If $x = [\![ u,v]\!]$, then by assumption $u = \phi(y)$ for some $y \in \hat{S}$ and $v = \phi(z)$ for some $z \in \hat{S}$, hence $x = \phi([\![ y,z]\!])$. If $x = \frak{s}(x_1 \cdots x_n) \rhd [y,z, w]$ we have by assumption $x_i = \phi(p_i)$, $y = \phi(u)$, $z= \phi(v)$ and $w = \phi(q)$ for some $p_i, u,v,q \in \hat{S}$ with $q = \nu \rhd r$, hence $x = \phi((\frak{s}(p_1 \cdots p_n)\cdot [u,v]\cdot  \nu) \rhd r)$. By the claim above and definition of $\T$, $\frak{s}(p_1 \cdots p_n)\cdot [u,v]\cdot  \nu$ is an OSBB-word of elements from $\hat{S}$, hence $(\frak{s}(p_1 \cdots p_n)\cdot  [u,v]\cdot  \nu) \rhd r\in \hat{S}$ which proves that $\phi(\hat{S}) = \T$. Since $\phi$ is an automorphism, this also proves that $\T$ is a basis of $A$. 
\end{proof}

We define an order $\prec_{\T}$ by $x \succ_{\T} y$ if $x = \phi(u)$, $y = \phi(v)$ and $u \succ_{\hat{S}} v$. We will also extend the definition of $b(\noarg)$ to $\T$ by $b(\phi(u)) = b(u)$ for any $u\in \hat{S}$.

\section{The Free post-Lie-Yamaguti Algebra}\label{sec:5}

Let $\PLY(\mathcal{C})$ be the free post-Lie-Yamaguti algebra generated by the set $\mathcal{C}$. Let $\I \subset A$ be the ideal generated by the post-Lie-Yamaguti algebra relations (i)-(v) in Definition \ref{def:InvConAlg}, that is, the ideal generated by the relations:
\begin{align}
    [\![ x, y ]\!] &= - [\![ y,x]\!], \tag{PLY1} \label{PLY:1} \\
    u \rhd [\![ x,y ]\!] &=  [\![u \rhd x , y ]\!] + [\![ x, u\rhd y ]\!], \tag{PLY2} \label{PLY:2} \\
    u \rhd [x,y,z] &= [u \rhd x,y,z] + [x,u \rhd y,z] + [x,y,u \rhd z] \tag{PLY3} \label{PLY:3} \\
    & \quad \, + (u \cdot [\![ x,y ]\!] ) \rhd z - [\![ x,y ]\!] \rhd (u \rhd z),  \nonumber \\
    \sum_{ \circlearrowleft (x,y,z)} [\![ [\![ x,y ]\!] , z ]\!] &= \sum_{ \circlearrowleft (x,y,z)}  [x,y,z]- [\![x,y ]\!] \rhd z, \tag{PLY4} \label{PLY:4} \\
    \sum_{ \circlearrowleft (x,y,z)} [ [\![ x,y ]\!] ,z,u] &= \sum_{ \circlearrowleft (x,y,z)} [\![ [\![ x,y]\!] , z]\!] \rhd u. \tag{PLY5} \label{PLY:5}
\end{align}
By definition $\PLY(\mathcal{C}) \cong A/\I$. The goal of this section is to provide a basis for $\PLY(\mathcal{C})$. This basis is in part described by LTS Hall-elements which gives a basis for the free Lie triple system, a generalization of the regular Hall-elements used to describe free Lie algebras. 

\subsection{Hall Elements}

Given a totally ordered set $\mathcal{C}$, a Hall set over $\mathcal{C}$ is a subset of the free magma $M(\mathcal{C})$. In this case the free magma will be the set generated by $\mathcal{C}$ and a free magmatic product denoted by $[\![ \noarg, \noarg ]\!]$. In the case $\mathcal{C} = \{a,b\}$ we have
\[
M(\{a,b\}) = \{ a, b, [\![a,a]\!] , [\![a,b]\!], [\![ b,a]\!] , [\![ b,b]\!], [\![ a, [\![ a,a]\!] ]\!], [\![ a, [\![ a, b]\!] ]\!], \ldots \}. 
\]
We remark that the free magma is the natural basis of the free algebra $\Alg(\mathcal{C})$, and the set $\OT(\mathcal{C})$ is an alternative basis for the vector space spanned by the free magma. In this article we choose to use the planar rooted trees representation exclusively for elements generated by the triangle product $\rhd$, while we will use this representation of bracketed words whenever the elements are generated by a bracket $[\![ \noarg, \noarg ]\!]$. There is of course no difference other than notation between the triangle and the bracket when considered as a free binary operator, and this preferred choice has no other justification than convenience.

Let $W(\mathcal{C})$ denote the set of all \textit{words} over the \textit{alphabet} $\mathcal{C}$, or more precisely; the set generated by $\mathcal{C}$ and an associative product. There is a natural surjective map $f$ from $M(\mathcal{C})$ to $W(\mathcal{C})$ that forgets the brackets and leave the elements of $\mathcal{C}$ in the given order. Let the length of a word be given by $l(c_1 c_2 \ldots c_n) =n$. 
\begin{definition}\label{def:Hall_order}
Let $\prec_H$ be the order on $W(\mathcal{C})$ given by 
\begin{itemize}
    \item $x \succ_H y$ if $l(x) > l(y)$, or
    \item $x \succ_H y$ if $l(x)=l(y)$, $x= x_1 \ldots x_n$, $y= y_1 \ldots y_n$ with $x_i = y_i$ for $i<j$ and $x_j \succ_{\mathcal{C}} y_j$.
\end{itemize}
\end{definition}

\begin{definition}[Hall-element]\label{def:Hall_element}
An element $x \in M(\mathcal{C})$ is an Hall element if
\begin{enumerate}[(i)]
    \item $x\in \mathcal{C}$, or
    \item $x = [\![u,v]\!]$ with $u,v$ both Hall elements with $f(u)\succ_H f(v)$ and
    \begin{itemize}
        \item $u\in \mathcal{C}$, or
        \item $u = [\![ w,z]\!]$ with $f(z) \preceq_H f(v)$.
    \end{itemize}
\end{enumerate}
\end{definition}
If $x \in M(\mathcal{C})$ is a Hall element, then $f(x) \in W(\mathcal{C})$ is a Hall word. The Hall set is a basis for the free Lie algebra generated by $\mathcal{C}$, see for instance \cite{Reutenauer1993free}. 
   

\subsection{Hall Elements of Lie Triple Systems}

Lie triple systems are closely related to Lie algebras; for any Lie triple system there is a canonical embedding into a $\mathbb{Z}_2$-graded Lie algebra. This was originally proved by Jacobson in \cite{Jacobson1951} and later improved by Yamaguti in \cite{yamaguti1957}. We need a type of Hall-elements to represent a basis for the free Lie triple system $\LTS(\mathcal{C})$ over a set $\mathcal{C}$, similar to the Hall-elements of the free Lie algebra. Let $M_3(\mathcal{C})$ be the free magma generated by a ternary product $[\![ \noarg, \noarg, \noarg ]\!]$ and the set $\mathcal{C}$. 
As with the free magma $M(\mathcal{C})$, we can associate a surjective map $f_3$ from $M_3(\mathcal{C})$ to $W(\mathcal{C})$ by forgetting the triple-brackets and leaving the elements in the given order. On $W(\mathcal{C})$ we have a total order given in Definition \ref{def:Hall_order}.

\begin{definition}[LTS Hall-element]
An element $x \in M_3(\mathcal{C})$ is a LTS Hall-element if
\begin{enumerate}[(i)]
    \item $x \in \mathcal{C}$, or
    \item $x = [\![ y,z,w]\!]$ with $y,z,w$ all LTS Hall-elements such that $f_3(y) \succ_H f_3(z)$, and either
    \begin{itemize}
        \item $y \in \mathcal{C}$, or
        \item $ y = [\![ u,v,t]\!]$ with $f_3(t) \succeq_H f_3(z)$.
    \end{itemize}
\end{enumerate}
\end{definition}

\begin{proposition}\label{prop:LTS-Hall}
    The LTS Hall-elements provide a basis for $\LTS(\mathcal{C})$. Given three LTS Hall-elements $u,v,w$, the following algorithm will rewrite $[\![u,v,w]\!]$ as a sum of LTS Hall-elements:
    \begin{enumerate}
        \item If $u=v$: $[\![u,v,w ]\!] \mapsto 0$. 
        \item If $f_3(u) \prec_H f_3(v)$: $[\![u,v,w]\!]\mapsto -[\![v,u,w]\!]$.
        \item If $f_3(u) \succ_H f_3(v) \succ_H f_3(w)$:  $[\![u,v,w]\!] \mapsto [\![u,w,v]\!]-[\![v,w,u]\!]$. 
        \item For $[\![u,v,w]\!]$ where $f_3(u) \succ_H f_3(v) \preceq_H f_3(w)$,  $u=[\![a,b,c]\!]$ and $f_3(c) \succ_H f_3(v)$:
        \begin{equation}
        [\![ [\![a,b,c]\!] ,v,w]\!] \mapsto [\![a,b,[\![c,v,w]\!] ]\!] - [\![c,[\![a,b,v]\!],w]\!] - [\![c,v,[\![a,b,w]\!] ]\!]. \label{eq:rewrite2}
        \end{equation}
        Repeat the algorithm until all terms are LTS Hall-elements. 
    \end{enumerate}
\end{proposition}

\begin{proof}
    See \cite[Section 4]{LAT_MKS_2023}.
\end{proof}


Let $\T_0$ be a subset of the basis $\T$ in Definition \ref{def:T_basis} defined by
\[
\T_0 = \{ x \in \T \, | \, x \neq [u,v,w] \text{ for any } u,v,w \in \T \}.
\]
In other words, $\T_0$ consist of all the elements in $\T$ that is not a triple-bracket. Notice that elements on the form $\frak{s}(x_1 \cdots x_n) \rhd [u,v,w]$ is in $\T_0$ as long as $n>0$. Consider the set $M_3(\T_0)$ with a ternary operator denoted by $[\![ \noarg, \noarg, \noarg ]\!]$. There is an injection 
\[
\begin{split}
    \iota : \T &\longrightarrow M_3(\T_0) \\
    \T_0 \ni x &\longmapsto x \\
    [u,v,w] &\longmapsto [\![ \iota(u), \iota(v), \iota(w)]\!]
\end{split}
\]
hence we have another total order on $\T$, which we will also denote $\prec_H$ given by $x \prec_H y$ if $f_3(\iota(x)) \prec_H f_3(\iota(y))$. Here we use the restriction of $\prec_{\T}$ to $\T_0$ when comparing elements in $\T_0$. 

\subsection{A Basis for the Free PLY Algebra}

\begin{definition}
    Let $\mathcal{B} \subset \T$ be defined by
    \begin{itemize}
        \item if $c \in \mathcal{C}$, then $c\in \mathcal{B}$,
        \item if $x,y \in \mathcal{B}$ with $x \succ_{H} y$, then $[\![ x,y ]\!] \in \mathcal{B}$,
        \item if $x,y,z \in \mathcal{B}$ such that
        \begin{enumerate}[(i)]
            \item $[x,y,z]$ is an LTS-Hall-element,
            \item if $x = [\![ u,v]\!]$, then $v\prec_H y$,
            \item if $y = [\![ u,v]\!]$, then $v \prec_H x$,
        \end{enumerate}
        then $[x,y,z] \in \mathcal{B}$.
        \item If $x_1 , \ldots , x_n \in \mathcal{B}$, then $\frak{s}(x_1 \cdots x_n) \rhd c \in \mathcal{B}$. 
    \end{itemize}
\end{definition}

\begin{theorem}\label{th:main}
    $\mathcal{B}$ is a basis for $\PLY(\mathcal{C})$.
\end{theorem}

\begin{proof}

First notice that $\mathcal{B} \subset \T$, hence the elements in $\mathcal{B}$ are all linearly independent in $A$. We will show that any element $x \in A$ can be rewritten as $x = x_1 +x_2$ where $x_1$ is in the span of $\mathcal{B}$ and $x_2$ is in the ideal $\I$. Without loss of generality we assume that $x \in \T$, and whenever convenient we let $x=\phi(\omega \rhd r)$ for some $\omega \rhd r \in \hat{S}$ where $\phi$ is the automorphism in Proposition \ref{prop:T_basis}. Below we give seven rewriting steps that show how any element $x$ can be written with one part in the span of $\mathcal{B}$ and the other part in $\I$. We use notation $x \mapsto x_1$ to mean that $x = x_1 +x_2$ with $x_2 \in \I$. 

If $|x|_V=1$, then $x \in \mathcal{B}$. Assume any $y \in \T$ with $|y|_V <k$ can be written as $y = y_1 + y_2$ with $y_1$ in the span of $\mathcal{B}$ and $y_2 \in \I$. Let $x \in \T$ with $|x|_V=k$.

\begin{enumerate}
    \item If $x = \phi(\omega \rhd [\![ u,v]\!])$, by Equation (\ref{eq:Dalg1}), (\ref{eq:Dalg2}) and (\ref{PLY:2}) we get
    \[
    \phi(\omega \rhd [\![ u,v]\!]) \longmapsto \sum_i [\![u_i, v_i]\!]. 
    \]
    Since $|u_i|_V<k$ and $|v_i|_V<k$, they can be written in the span of $\mathcal{B}$ added to something in $\I$ by assumption. 
    \item If $x = [\![ u,v]\!]$ with $u,v \in \mathcal{B}$ and $u \prec_H v$, by Equation (\ref{PLY:1}) we get
    \[
    [\![ u,v]\!] \longmapsto - [\![v,u]\!].
    \]
    In the case $u=v$, we get $[\![ u,v]\!] \mapsto 0$.
\end{enumerate}
After the first two steps, any $x = \phi(\omega \rhd r)$ with $r = [\![ u,v]\!]$ will be rewritten $x \mapsto x_1$ where $x_1$ is in the span of $\mathcal{B}$ and $x-x_1 \in \I$. We still need to give a rewriting for the elements on the form $x = \phi(\omega \rhd c)$ where $c\in \mathcal{C}$. If $y\in \mathcal{B}$, then $y \rhd c \in \mathcal{B}$. Assume that whenever $l(\eta) < l(\omega)$ we have a rewriting $\phi(\eta \rhd c) \mapsto y_1$ in the span of $\mathcal{B}$ with $\phi(\eta \rhd c) - y_1 \in \I$. 
\begin{enumerate}
    \setcounter{enumi}{2}
    \item If $x = \phi(\omega \rhd c) = \frak{s}(x_1 \cdots x_n) \rhd [u,v,w]$ where $x_1, \ldots ,x_n, u,v,w \in \mathcal{B}$, by Equation (\ref{eq:Dalg1}), (\ref{eq:Dalg2}) and (\ref{PLY:3}) we get
    \[
    \frak{s}(x_1 \cdots x_n) \rhd [u,v,w] \longmapsto \sum_i [u_i,v_i,w_i] + \sum_j \phi(\eta_j \rhd c)
    \]
    where $l(\eta_j)<l(\omega)$, hence these terms can be split into a part in the span of $\mathcal{B}$ and a part in $\I$ by assumption.
    \item If $x = [u,v,w]$ with $u,v,w \in \mathcal{B}$ and $u=v$, then $[u,v,w] \mapsto 0$ by the definition of the triple-bracket. If $u \prec_H v$, again by the definition of the triple-bracket we get
    \[
    [u,v,w] \longmapsto - [v,u,w].
    \]
    \item If $x = [u,v,w]$ with $u,v,w \in \mathcal{B}$ and $u \succ_H v \succ_H w$, by Equation (\ref{PLY:4}) we get
    \[
    [u,v,w] \longmapsto [u,w,v] - [v,w,u] + \sum_{\circlearrowleft (u,v,w)}[\![  [\![ u,v]\!] ,w]\!] + [\![u,v]\!] \rhd w.
    \]
    Here the last terms have increase in $b(\noarg)$ compared to $x$.
    \item If $x = [[y,z,u],v,w]$ with $y,z,u,v,w \in \mathcal{B}$ and $u \succ_H v$, by Equation (\ref{PLY:6}) we get
    \[
    \begin{split}
        [[y,z,u],v,w] \longmapsto &[y,z,[u,v,w]] - [u,[y,z,v],w] - [u,v,[y,z,w]] \\
        &- ([y,z]\cdot [\![u,v]\!]) \rhd w + [\![ u,v]\!] \rhd [y,z,w]
    \end{split} 
    \]
    The last terms have increase in $b(\noarg)$ compared to $x$.
\end{enumerate}
For any term $y$ with $b(y)=b(x)$, if $y$ is not LTS-Hall, repeat steps 4 to 6. By Proposition \ref{prop:LTS-Hall} this process must terminate with only LTS-Hall-elements and terms $z$ with $b(z)>b(x)$. 
\begin{enumerate}
    \setcounter{enumi}{6}
    \item If $x = [[\![ u,v]\!],z,w]$ with $u,v,z,w\in \mathcal{B}$ and $v \succ_H z$, by Equation (\ref{PLY:5}) we get
    \[
    \begin{split}
        [ [\![ u,v]\!] ,z ,w ] \longmapsto & [ [\![u,z]\!], v,w] - [ [\![v,z]\!],u,w ] + \sum_{\circlearrowleft (u,v,z) } [\![ [\![ u,v]\!] ,z]\!] \rhd w.
    \end{split}
    \]
    Similarly, if $x = [z, [\![ u,v]\!] , w]$ with $v \succ_H z$, by the same equation we get
    \[
    [z, [\![ u,v]\!] , w] \longmapsto  - [ [\![u,z]\!], v,w] + [ [\![v,z]\!],u,w ] - \sum_{\circlearrowleft (u,v,z) } [\![ [\![ u,v]\!] ,z]\!] \rhd w.
    \]
    The last term has increase in $b(\noarg)$ compared to $x$. The other two terms will have brackets $[\![ \noarg,z ]\!]$ where $z \prec_H v$, hence the last term of the bracket is strictly smaller by the order $\prec_H$.
\end{enumerate}
At this point, repeat the all the steps for all the terms that is not already in $\mathcal{B}$. Notice that these terms will either have $b(y)>b(x)$ or there will be a bracket $[\![u,v]\!]$ as a factor of $x$ that is replaced by a bracket $[\![ \noarg, z]\!]$ with $z \prec_H v$. Repeat all the steps for these elements. Since $b(x)$ has an upper bound at $|x|_V -1$ and there is only a finite number of elements $z$ satisfying $|z|_V<k$ and $z \prec_H v$ for a given $v$ in a factor $[\![u,v]\!]$ of $x$, the process must terminate. At that point we have rewritten $x$ as $x_1 +x_2$ with $x_1$ in the span of $\mathcal{B}$ and $x_2\in \I$, and by induction on $|x|_V$, this is true for any $x \in \T$. \\

We still need to prove that if $x \neq 0$ is in the span of $\mathcal{B}$, then $x$ is not in $\I$. Assume $x\neq 0$ is both in the span of $\mathcal{B}$ and in $\I$, and assume that there is no elements $y$ with $|y|_V <|x|_V$ satisfying this. Then $x$ must be on the form of one of the defining relations of $\I$. 
\begin{enumerate}[(i)]
    \item $x$ can't be generated from Equation (\ref{PLY:1}) since $[\![ u,v]\!] \in \T$ and $[\![ v,u]\!] \in \T$, but both can't be in the span of $\mathcal{B}$.
    \item $x$ can't be generated from Equation (\ref{PLY:2}) since $y \rhd [\![ u,v]\!] \in \T$ can't be in the span of $\mathcal{B}$.
    \item Equation (\ref{PLY:3}) is a bit more complex than the other cases. Certainly there can't be any terms on the form $y \rhd [u,v,w]$ in the span of $\mathcal{B}$, hence there are no elements directly generated from this equation in the span of $\mathcal{B}$. However, Equation (\ref{PLY:6}) is a direct consequence of Equation (\ref{PLY:3}), and elements on the form $[y,z,[u,v,w]]$ are contained in $\mathcal{B}$. If $[y,z,[u,v,w]]\in \mathcal{B}$ we have $u \succ_H v$, and then the term $[[y,z,u],v,w] \notin \mathcal{B}$. This is a term in Equation (\ref{PLY:6}), hence $x$ can't be generated from Equation (\ref{PLY:3}). Note that $y \rhd [u,v,w] \in \T$ and $[[y,z,u],v,w] \in \T$.
    \item $x$ can't be generated from Equation (\ref{PLY:4}) since the cyclic sum of $[u,v,w]$ are all elements in the basis $\T$ after adjusting for skew-symmetry, but only two of those can be in the span of $\mathcal{B}$.
    \item $x$ can't be generated from Equation (\ref{PLY:5}) since the cyclic sum \\
    $\sum_{ \circlearrowleft (u,v,z) } [[\![u,v]\!] ,z,w]$ are all elements of $\T$ after adjusting for skew-symmetry in the triple-bracket, but one of these terms can't be in the span of $\mathcal{B}$.
\end{enumerate}
Since any element we are able to generate in the ideal $\I$ necessarily contains a term in the basis $\T$ that is not contained in the span of $\mathcal{B}$, we conclude that the intersection between the span of $\mathcal{B}$ and $\I$ is $\{ 0\}$. Thus the span of $\mathcal{B}$ is equal to $A/\I \cong \PLY(\mathcal{C})$, and since $\mathcal{B}$ is a linearly independent set it follows that $\mathcal{B}$ is a basis for the free post-Lie-Yamaguti algebra.

\end{proof}

We list all the basis elements of $\mathcal{B}\subset \PLY(\{\ab\})$ with $|x|_V\leq 5$ in increasing order with respect to $\prec_{\T}$. 
\[
\begin{split}
    \textbf{1 vertex:}& \quad \ab, \\
    \textbf{2 vertices:}& \quad \aabb, \\
    \textbf{3 vertices:}& \quad [\![ \aabb , \ab ]\!], \,\, \aaabbb, \,\, \aababb, \\
    \textbf{4 vertices:}& \quad   [\![ [\![ \aabb , \ab ]\!] , \ab ]\!], \,\, [\![ \aaabbb , \ab ]\!], \,\, [\![ \aababb , \ab ]\!], \,\, [\![ \aabb, \ab ]\!] \rhd \ab, \,\, \aaaabbbb, \,\, \aaababbb, \,\, [\aabb, \ab ,  \ab], \,\, \frak{s}(\aabb \, \ab) \rhd \ab, \,\,  \aabababb, \\
    \textbf{5 vertices:}& \quad  [\![ [\![ [\![ \aabb , \ab ]\!], \ab ]\!] , \ab ]\!], \,\, [\![ [\![ \aabb , \ab ]\!], \aabb]\!] , \,\, [\![ [\![ \aaabbb , \ab ]\!] , \ab ]\!] , \,\, [\![ [\![ \aababb , \ab ]\!] , \ab ]\!] , \,\, [\![ \aaabbb, \aabb ]\!] , \,\, [\![ \aababb , \aabb ]\!] , \\ %
    & \quad [\![  [\![ \aabb , \ab ]\!]  \rhd \ab , \ab ]\!] , \,\, [\![ \aaaabbbb , \ab ]\!] , \,\, [\![ \aaababbb, \ab ]\!] , \,\, [\![ [\aabb, \ab , \ab] , \ab ]\!] , \,\, [\![ \frak{s}(\aabb \, \ab ) \rhd \ab , \ab ]\!] , \,\, [\![ \aabababb, \ab ]\!] , \\ %
    & \quad  [\![ [\![ \aabb , \ab ]\!] , \ab ]\!]  \rhd \ab, \,\, [\![ \aaabbb, \ab ]\!]  \rhd \ab , \,\, [\![ \aababb, \ab ]\!]  \rhd \ab , \,\, ( [\![ \aabb , \ab ]\!]  \rhd \ab ) \rhd \ab , \,\,  \aaaabbbb  \rhd \ab , \,\, \aaababbb \rhd \ab , \\ %
    & \quad  [\aabb , \ab , \ab ] \rhd \ab , \,\, ( \frak{s}(\aabb \, \ab ) \rhd \ab ) \rhd \ab, \,\,  \aabababb  \rhd \ab , \,\, 
    [ [\![ \aabb, \ab ]\!] , \ab , \ab] , \,\, [\aaabbb, \ab ,\ab ], \,\, [\aababb, \ab ,\ab] ,\,\,  \\ %
    & \quad  (\aabb \, \aabb ) \rhd \ab , \,\, \frak{s}([\![ \aabb, \ab ]\!] \, \ab ) \rhd \ab , \,\, \frak{s}(\aaabbb \, \ab ) \rhd \ab , \,\, \frak{s}(\aababb \, \ab ) \rhd \ab , \,\,  [\aabb, \ab ,\aabb], \,\,  \frak{s}(\aabb \, \ab \, \ab) \rhd \ab ,  \\
    \vspace{0cm}\\
    & \quad \aababababb.
 \end{split}
\]

\subsection{Relations to the basis of $\LAT(\mathcal{C})$}
    In \cite{LAT_MKS_2023} a basis $\mathcal{B}_{\LAT}$ for the free Lie admissible triple algebra is given by
    \begin{itemize}
        \item if $c \in \mathcal{C}$, then $c \in \mathcal{B}_{\LAT}$,
        \item if $x_1 , \ldots , x_n \in \mathcal{B}_{\LAT}$, then $\frak{s}(x_1 \cdots x_n) \rhd c \in \mathcal{B}_{\LAT}$, and
        \item if $x,y,z \in \mathcal{B}_{\LAT}$ and $[x,y,z]$ is an LTS Hall-element, then $[x,y,z] \in \mathcal{B}_{\LAT}$. 
    \end{itemize}
    If we let $[\![ \noarg , \noarg ]\!] =0$ in $\mathcal{B}$, the remaining elements form exactly $\mathcal{B}_{\LAT}$, hence the basis $\mathcal{B}$ is a natural generalization of $\mathcal{B}_{\LAT}$. More precisely; the following diagram commutes
    \[
    \begin{tikzcd}
    \Alg(\mathcal{C}; [\![ \noarg, \noarg ]\!] , \rhd) \arrow{r}{[\![\noarg,\noarg]\!] \mapsto 0} \arrow[swap]{d}{\pi_{\mathcal{B}}} & \Alg(\mathcal{C}; \rhd) \arrow{d}{\pi_{\mathcal{B}_{\LAT}}} \\%
    \PLY(\mathcal{C}) \arrow{r}{[\![\noarg,\noarg]\!] \mapsto 0}& \LAT(\mathcal{C})
    \end{tikzcd}
    \]
    where $\Alg(\mathcal{C}; [\![ \noarg, \noarg ]\!] , \rhd)$ is the free algebra over two binary operations, $\Alg(\mathcal{C}; \rhd)$ is the free algebra over one binary operation, $\pi_{\mathcal{B}}$ is the projection to the span or $\mathcal{B}$ and $\pi_{\mathcal{B}_{\LAT}}$ is the projection to the span of $\mathcal{B}_{\LAT}$.

    It should be noted that choices were made when constructing the basis $\mathcal{B}$ that made it fit with the basis of the free LAT. The basis $\mathcal{B}$ does not reflect the basis of the free post-Lie algebra in the same way. In particular in step 5 the rewriting algorithm in the proof of Theorem \ref{th:main} we chose to interpret Equation (\ref{PLY:4}) as a weaker Jacobi identity on the triple-bracket $[x,y,z]$, and thus making sure the elements in the basis are always LTS Hall-elements. If we instead interpreted this equation as a weaker Jacobi identity on $[\![ [\![ x,y]\!] ,z]\!]$, we could potentially rewrite all the terms in the basis such that every bracket $[\![x,y]\!]$ is a Hall-element. In this case the resulting basis should generalize the basis of the free post-Lie algebra with a natural projection given by $[\![ x,y]\!] \rhd z \mapsto [x,y,z]$. We were not able to provide such a basis in this article because it was more complicated to show that the rewriting algorithms would terminate.

\section{Lie-Yamaguti Algebras}\label{sec:6}

Lie-Yamaguti algebras are an algebraic structure that appears in the vector fields on reductive homogeneous spaces. Thus it is expected that this structure is related to post-Lie-Yamaguti algebras. This relation will be made precise in Proposition \ref{prop:PLY_LY}

\begin{definition}[Lie-Yamaguti]
    Let $\mathcal{A}$ be an algebra with a binary operation $\circ$ and a ternary operation $\{ \noarg, \noarg, \noarg \}$. $(\mathcal{A}, \circ, \{ \noarg, \noarg, \noarg \})$ is called a Lie-Yamaguti algebra if it satisfies
    \begin{align}
        x\circ x &= 0, \tag{LY1} \label{LY1} \\
        \{x,x,y\} &= 0, \tag{LY2} \label{LY2} \\
        \sum_{ \circlearrowleft (x,y,z)} \{x,y,z\} + (x\circ y) \circ z &= 0, \tag{LY3} \label{LY3}\\
        \sum_{\circlearrowleft (x,y,z) } \{ x \circ y,z,w \} &= 0, \tag{LY4} \label{LY4} \\
        \{ x,y,z\circ w \} &= \{x,y,z \} \circ w + z \circ \{x,y,w\}, \tag{LY5} \label{LY5} \\
        \{x,y,\{u,v,w\} \} &= \{ \{x,y,u\} ,v,w\} + \{u, \{x,y,v\}, w\} + \{ u,v,\{x,y,w\} \}. \tag{LY6} \label{LY6}
    \end{align}
\end{definition}

Lie-Yamaguti algebras were   first introduced by Yamaguti in \cite{Yamaguti1958} under the name \textit{generalized Lie triple systems}. In that article it Yamaguti shows that if $(M,\nabla)$ is a reductive homogeneous space with its canonical connection, then $(\mathfrak{X}_M, -T, -R)$ is a Lie-Yamaguti algebra, where $T$ and $R$ is the torsion and curvature. Lie-Yamaguti algebras have been studied in relation to reductive homogeneous spaces in for instance \cite{kikkawa1975geometry,Kinyon_Weinstein_2001,benito2009irreducible}. 

\begin{proposition}\label{prop:PLY_LY}
    Let $(\mathcal{A}, [\![ \noarg, \noarg ]\!], \rhd )$ be a post-Lie-Yamaguti algebra. Define \\
    $\{x,y,z \} := [\![ x,y ]\!] \rhd z - [x,y,z]$. Then $(\mathcal{A}, [\![ \noarg, \noarg ]\!] , \{ \noarg, \noarg, \noarg \})$ is a Lie-Yamaguti algebra.
\end{proposition}

\begin{proof}
We verify the six defining properties of Lie-Yamaguti algebras.
\begin{itemize}
    \item \eqref{LY1} follows immediately from \eqref{PLY:1}.
    \item \eqref{LY2} follows from the definition of the triple-bracket in addition to the skew-symmetry of $[\![ \noarg,\noarg]\!]$.
    \item \eqref{LY3} follows from \eqref{PLY:4}:
    \[
    \begin{split}
        \sum_{ \circlearrowleft (x,y,z)} \{x,y,z\} + [\![ [\![x,y]\!] ,z ]\!] &= \sum_{ \circlearrowleft (x,y,z)}  [\![x,y]\!] \rhd z - [x,y,z] + [\![ [\![x,y]\!] ,z ]\!] =0.
    \end{split}
    \]
    \item \eqref{LY4} follows from \eqref{PLY:5}:
    \[
    \sum_{\circlearrowleft (x,y,z) } \{[\![x,y]\!],z,w \} = \sum_{\circlearrowleft (x,y,z) } [\![ [\![ x,y]\!] z]\!] \rhd w -  [[\![x,y]\!],z,w ] = 0.
    \]
    \item \eqref{LY5} follows from \eqref{PLY:2}:
    \[
    \begin{split}
        \{ x,y,z\circ w \} &= [\![x,y]\!] \rhd [\![z,w]\!] - [x,y,[\![z,w]\!] ] \\
        &= [\![ [\![x,y]\!] \rhd z ,w]\!] + [\![ z, [\![x,y]\!] \rhd w]\!] - [\![ [x,y,z] , w]\!] - [\![ z,[x,y,w]]\!] \\
        &= [\![ \{x,y,z \} , w ]\!] + [\![ z, \{x,y,w\} ]\!],
    \end{split}
    \]
    where the fact that $[x,y, \noarg]$ acts as a derivation on $[\![ z,w]\!]$ follows from a straight forward computation using \eqref{PLY:2} and the definition of the triple-bracket \eqref{eq:TB}.
    \item For \eqref{LY6} we use \eqref{PLY:3} and \eqref{PLY:6} in a lengthy computation:
    \[
    \begin{split}
        \{ x,y, \{ u,v,w\} \} &= [\![ x,y ]\!] \rhd ([\![u,v]\!] \rhd w) - [\![ x,y ]\!] \rhd  [u,v,w] \\
        & \quad - [ x,y, [\![ u,v]\!] \rhd w] + [x,y,[u,v,w]] \\
        &= ([\![ x,y ]\!] \cdot [\![u,v]\!]) \rhd w + [\![ [\![ x,y]\!] \rhd u , v]\!] \rhd w \\
        & \quad + [\![ u, [\![ x,y]\!] \rhd v]\!] \rhd w - [ [\![ x,y]\!] \rhd u, v,w ] \\
        & \quad - [u , [\![x,y]\!] \rhd v, w] - [u,v,[\![ x,y]\!] \rhd w] \\
        & \quad - ([\![ x,y]\!] \cdot [\![ u,v]\!] ) \rhd w + [\![ u,v]\!] \rhd ( [\![ x,y]\!] \rhd w) \\
        & \quad - [ x,y, [\![ u,v]\!] \rhd w] + [ [x,y,u] , v,w] \\
        & \quad + [u, [x,y,v],w] + [u,v,[x,y,w]] \\
        & \quad + ([x,y] \cdot [\![ u,v]\!] ) \rhd w - [\![ u,v]\!] \rhd [x,y,w] \\
        &= [\![ [\![ x,y]\!] \rhd u , v]\!] \rhd w + [\![ u, [\![ x,y]\!] \rhd v]\!] \rhd w + [\![ u,v]\!] \rhd ( [\![ x,y]\!] \rhd w) \\
        & \quad - [ [\![ x,y]\!] \rhd u, v,w ] - [u , [\![x,y]\!] \rhd v, w] - [u,v,[\![ x,y]\!] \rhd w] \\
        & \quad + \Big( ([x,y] \cdot [\![ u,v]\!] ) \rhd w - [x,y, [\![ u,v]\!] \rhd w] \Big) - [\![ u,v]\!] \rhd [x,y,w]. \\
        & \quad + [ [x,y,u] , v,w] + [u, [x,y,v],w]  + [u,v,[x,y,w]] \\
        &= \{ \{x,y,u\} ,v,w\} + \{u, \{x,y,v\}, w\} + \{ u,v,\{x,y,w\} \}
    \end{split}
    \]
    where the terms in the last line sums up the terms in each column above it after rewriting the terms in parentheses by
    \[
    \begin{split}
         ([x,y] \cdot [\![ u,v]\!] ) \rhd w - [x,y, [\![ u,v]\!] \rhd w] &= [x,y, [\![ u,v]\!]  \rhd w] - [x,y,[\![ u,v]\!] ]\rhd w \\
         & \quad \quad \quad - [x,y, [\![ u,v]\!] \rhd w]\\
        &= -[\![ [x,y,u],v]] \rhd w - [\![ u, [x,y,v] ]\!] \rhd w.
    \end{split} 
    \]
\end{itemize}
\end{proof}

In relation to a reductive homogeneous space the bracket $[\![\noarg ,\noarg ]\!]$ in the PLY algebra and the binary operation $\circ$ in the Lie-Yamaguti algebra will both be defined as the negative of the torsion. One major difference is that the Lie-Yamaguti algebra associated to a reductive homogeneous space will be tensorial in both its operators, while the PLY algebra is only tensorial in the bracket $[\![ \noarg, \noarg]\!]$. 

\begin{remark}
    Consider a post-Lie algebra $(\mathcal{A}, [\noarg,\noarg], \rhd )$. By Remark \ref{rmrk1} $\mathcal{A}$ is also a PLY algebra, hence we can create a Lie-Yamaguti algebra as in Proposition \ref{prop:PLY_LY}. In this case we get $\{ \noarg,\noarg,\noarg \} = 0$ by the second post-Lie relation, and we are left with just the Lie algebra $(\mathcal{A}, [\noarg, \noarg])$.

    Similarly, if we consider a LAT algebra $(\mathcal{A} , \rhd)$ and interpret this as a PLY algebra, we can use Proposition \ref{prop:PLY_LY} again to create a Lie-Yamaguti algebra. In this case the binary operation will vanish and the ternary operation will be the triple-bracket $\{\noarg,\noarg,\noarg \} = -[\noarg,\noarg, \noarg ]$. We are left with only the Lie triple system $(\mathcal{A}, -[\noarg,\noarg,\noarg])$. 
\end{remark}

\section*{Acknowledgement}
This article is dedicated to Hans Munthe-Kaas and Brynjulf Owren in regards to their 60th birthdays. Prof. Munthe-Kaas has offered great guidance to the author over the last four years and none of the research presented in this paper would be possible without him. Prof. Owren has generously provided opportunities for the author to present his research, and the author wants to thank him for all his friendly encouragements. In addition, the author wants to thank Prof. Kurusch Ebrahimi-Fard who encouraged the author to write this article, and Adrien Laurent for his ever willingness to help the author with his research. Lastly, the author wants to thank the anonymous referees for their relentless efforts; their valuable feedback has been crucial for finalizing this product. The author is supported by the Research Council of Norway through project 302831 Computational Dynamics and Stochastics on Manifolds (CODYSMA).

\bibliographystyle{abbrv}

\end{document}